\documentclass[11pt]{article}
\usepackage{amsthm}
\usepackage{amssymb}
\usepackage{amsmath}
\usepackage{hyperref}
\usepackage{graphicx}
\usepackage{caption}
\usepackage{array}
\usepackage{enumitem}
\usepackage{mathtools}
\usepackage{etoolbox}
\usepackage[backend=biber, style=alphabetic, maxbibnames=9, maxcitenames=9, maxalphanames=9]{biblatex}
\bibliography{ConstMonodromy}
\usepackage[all,cmtip]{xy}

\pretocmd{\section}{%
  }{}{}

\numberwithin{table}{section}

\newtheorem{theorem}{Theorem}[section]
\newtheorem*{theorem*}{Theorem}
\newtheorem{proposition}[theorem]{Proposition}
\newtheorem*{proposition*}{Proposition}
\newtheorem{corollary}[theorem]{Corollary}
\newtheorem{lemma}[theorem]{Lemma}
\newtheorem*{lemma*}{Lemma}

\newtheorem{introthm}{Theorem}[section]

\theoremstyle{definition}
\newtheorem{definition}[theorem]{Definition}

\newtheorem*{exercise*}{Exercise}
\newtheorem{remark}[theorem]{Remark}

\newtheorem{notation}[theorem]{Notation}

\numberwithin{equation}{section}
\numberwithin{figure}{section}

\let\hom\relax
\let\det\relax

\newcommand{\mc}{\mathcal}
\newcommand{\mf}{\mathfrak}
\newcommand{\mr}{\mathrm}

\newcommand{\Z}{\mathbb{Z}}
\newcommand{\Q}{\mathbb{Q}}
\newcommand{\R}{\mathbb{R}}
\newcommand{\C}{\mathbb{C}}

\newcommand{\A}{\mathbb{A}}
\renewcommand{\O}{\mathcal{O}}

\newcommand{\Sp}{\mathrm{Sp}}

\newcommand{\sset}[2]{\lbrace{#1}\,\,|\,\,{#2}\rbrace}

\DeclareMathOperator{\chars}{char}

\DeclareMathOperator{\cond}{cond}

\DeclareMathOperator{\det}{det}
\DeclareMathOperator{\diag}{diag}

\DeclareMathOperator{\ev}{ev}

\DeclareMathOperator{\hom}{Hom}
\newcommand{\id}{\mathrm{id}}

\DeclareMathOperator{\Ind}{Ind}

\DeclareMathOperator{\Jac}{Jac}

\newcommand{\modulo}[1]{\,\,(\mathrm{mod}\,\,{#1})}
\DeclareMathOperator{\mspec}{mspec}

\DeclareMathOperator{\rec}{rec}

\DeclareMathOperator{\sss}{ss}

\DeclareMathOperator{\St}{St}

\DeclareMathOperator{\Supp}{Supp}

\DeclareMathOperator{\tr}{Tr}
\DeclareMathOperator{\triv}{triv}

\DeclareMathOperator{\vol}{Vol}

\begin{document}
\title{On the Bernstein--Zelevinsky classification and epsilon factors in families}
\author{Sam Mundy}
\date{}
\maketitle

\begin{abstract}
We define a certain Hecke theoretic notion of family of smooth admissible representations of $GL_n(F)$, or of products of such groups, where $F$ is a nonarchimedean local field of characteristic zero. While this notion of family is rather weak a priori, we show that it implies strong rigidity properties for the Bernstein--Zelevinsky presentations of the members of such families, as well as for the variation of their epsilon factors (attached to arbitrary functorial lifts).

Examples of such families come often from the theory of eigenvarieties, and in this case our results imply analyticity properties for the epsilon factors of the members of these families.
\end{abstract}

\section*{Introduction}

The goal of this paper is to study the variation of smooth admissible representations of $p$-adic $GL_n$ in families. The notion of family we consider is a Hecke theoretic one, examples of which come from the theory of eigenvarieties. To fix ideas and state our results, let us start by giving the precise definition of such a family.

Let $F$ be a nonarchimedean local field of characteristic zero, and let $n>0$ be an integer. Write
\[\mc{H}=C_c^\infty(GL_n(F),\C)\]
for the Hecke algebra of $GL_n(F)$. Then by a \textit{family of smooth admissible representations of} $GL_n(F)$ we mean a quadruple $\mc{F}=(X,\Sigma,R,T)$ where:
\begin{itemize}
\item $X$ is a topological space,
\item $\Sigma$ is a dense subset of $X$,
\item $R$ is a $\C$-subalgebra of the ring of $\C$-valued functions on $X$,
\item $T$ is a $\C$-linear map $\mc{H}\to R$,
\end{itemize}
which satisfies the following properties:
\begin{enumerate}[label=(\roman*)]
\item For any $a\in\C$ and $\phi\in R$, the set $\phi^{-1}(a)\subset X$ is closed in $X$,
\item For any $x\in\Sigma$, there exists a smooth admissible representation $\pi_x$ of $GL_n(F)$ such that for any $f\in\mc{H}$, the composition $T_x=\ev_x\circ T$ of $T$ with the evaluation map $\ev_x:R\to\C$ of functions in $R$ at $x$ satisfies $T_x(f)=\tr(f|\pi_x)$.
\end{enumerate}
Moreover, we call $\mc{F}$ \textit{generically irreducible} if for all $x\in\Sigma$, the representations $\pi_x$ of condition (ii) above can be taken to be irreducible.

We remark that in the main body of this paper, we actually consider families of representations for products of several groups $GL_n(F)$ for varying $n$ and $F$; this is because families for such products are what arise naturally in the theory of eigenvarieties. We refer to Section \ref{secremeigen} at the end of this paper for how such families arise in this way. However, for simplicity, we state our main results for a single group $GL_n(F)$ in this introduction. Actually, the proofs of the main results in the body of this paper will go by reducing to the case of a single $GL_n(F)$.

The coherence of these families comes from the condition (i) of the definition. Although the continuity property required by this condition (which is simply that the functions in $R$ are continuous for the cofinite topology on $\C$) seems rather weak, our first main result is that this is already enough to rigidify the inertial types of the members of the family. To state this result precisely, we assume that the reader is familiar with the Bernstein--Zelevinsky classification; see Section \ref{secBZ} for a complete review of this classification. We just note here that our formulation of the classification uses segments
\[\Delta=\{\tau,\tau\otimes\vert\det\vert,\dotsc,\tau\otimes\vert\det\vert^{\ell-1}\},\]
where the powers of $\vert\det\vert$ are in increasing order, and is done in terms of quotients $Q(\Delta)$ and $Q(\Delta_1,\dotsc,\Delta_r)$ of various parabolic inductions. Thus such representations $Q(\Delta)$ are generalized Steinberg representations.

The following theorem is our Theorem \ref{thmfamilythm} in the case of a single group $GL_n(F)$.

\begin{introthm}
\label{ithmfamilythm}
Let $\mc{F}=(X,\Sigma,R,T)$ be a generically irreducible family of smooth admissible representations of $GL_n(F)$. For $x\in\Sigma$, let $\pi_x$ be the irreducible smooth representation of $GL_n(F)$ such that $T_x(f)=\tr(f|\pi_x)$. Fix $x_0\in\Sigma$ and write
\[\pi_{x_0}\cong Q(\Delta_1,\dotsc,\Delta_r)\]
in the Bernstein--Zelevinsky classification, for some segments $\Delta_1,\dotsc,\Delta_r$. Then there is an open and closed neighborhood $X_0$ of $x_0$ in $X$ with the following property: For any $x\in \Sigma\cap X_0$, there are unramified characters $\chi_{k,x}$ of $F^\times$, $1\leq k\leq r$, such that
\[\pi_x\cong Q(\Delta_1\otimes\chi_{1,x},\dotsc,\Delta_r\otimes\chi_{r,x}).\]
Here, for any $k$ with $1\leq k\leq r$, the segment $\Delta_k\otimes\chi_{k,x}$ denotes the segment whose members are just those of $\Delta_k$ each twisted by $\chi_{k,x}\circ\det$.
\end{introthm}

Although the proof of this theorem makes no essential use of Weil--Deligne representations, its main ideas are perhaps best described by passing through the local Langlands correspondence (which we briefly review in this paper after Corollary \ref{corpfineq}). For each $k$ with $1\leq k\leq r$, Let $\rho_k$ be the irreducible representation of the Weil group $W_F$ of $F$ attached to the first supercuspidal representation in $\Delta_k$. Let $\ell_k$ be the length of $\Delta_k$, and $\Sp(\ell_k)$ the special Weil--Deligne representation of dimension $\ell_k$. Then $\pi_{x_0}$ has associated Weil--Deligne representation
\[(\rho,N)=\bigoplus_{k=1}^r\rho_k\otimes\Sp(\ell_k).\]

Now the type theory from the work of Schneider and Zink \cite{SZ} gives us an idempotent $e$ in the Hecke algebra $\mc{H}$ such that, for any irreducible smooth representation $\pi'$ of $GL_n(F)$ with associated Weil--Deligne representation $(\rho',N')$, if the operator $e$ acts in a nonzero way on $\pi'$, then $\rho|_{I_F}\cong\rho'|_{I_F}$, and the adjoint orbit of $N'$ has dimension at most that of $N$; here $I_F$ denotes the inertia group of $F$. (Strictly speaking, to deduce this inequality of dimensions from what is in \cite{SZ}, one must use Corollary \ref{corpfineq} in this paper.) We deduce from the coherence condition (i) in our definition of family that $T(e)$ is locally constant in $x$.

This reduces the proof to showing that the dimensions of the orbits of the monodromy operators, say $N_x$, for the Weil--Deligne representations attached to $\pi_x$ for $x$ near $x_0$ cannot increase. Using Arthur--Clozel local base change \cite{AC}, we further reduce to the case when $\rho|_{I_F}$ is trivial. Then $\pi_{x_0}$ has Iwahori fixed vectors, and hence fixed vectors by the subgroup
\[K_{1,n}=\sset{g\in GL_n(\O_F)}{g\equiv 1\modulo{\varpi}},\]
where $\O_F$ is the valuation ring of $F$ and $\varpi\in\O_F$ is a uniformizer. To finish the proof, we then link the size of the orbits of $N_x$ with the function $T(\chars(K_{1,n}))$ using the following (reformulation of) Theorem \ref{thmK1fixed} in this paper (see Corollary \ref{cormonodcount}).

\begin{introthm}
\label{ithmK1fixed}
Let $\pi$ be an irreducible smooth representation of $GL_n(F)$ with associated Weil--Deligne representation $(\rho,N)$. Assume $\rho|_{I_F}$ is trivial. Let $v_1,\dotsc,v_n$ be a basis of the space $V$ of $(\rho,N)$ in which $N$ is written in Jordan form. Then
\[\frac{v_p(\dim_{\C}(\pi^{K_{1,n}}))}{v_p(\vert\O_F/\varpi\vert)}\]
is the number of nonzero entries of the matrix of $\exp(N)-\id_V$ in the basis $v_1,\dotsc,v_n$. In the formula, $v_p$ denotes $p$-adic valuation, where $p$ is the characteristic of the residue field of $F$.
\end{introthm}

The proof of this theorem is not too long and goes by a comparison of the Steinberg representation of $GL_n(F)$ with its finite field analogue, as well as comparisons of character formulas involving them.

We then apply Theorem \ref{ithmfamilythm} in the proof of the following theorem, which is a special case of Theorem \ref{thmepsilon}, in the case of one group $GL_n(F)$ and one family $\mc{F}$.

\begin{introthm}
\label{ithmepsilon}
Let $\mc{F}=(X,\Sigma,R,T)$ family of smooth admissible representations of $GL_n(F)$. Let $s\in\C$, and let $\psi$ be a nontrivial additive character of $F$. Let $\varrho$ be an algebraic representation of $GL_n(\C)$. For any $x\in\Sigma$, let $\pi_x$ be the representation of $GL_n(F)$ so that $T_x(f)=\tr(f|\pi_x)$ for any $f\in\mc{H}$. Finally, assume that for some $x_0\in\Sigma$, no two of the segments in the Bernstein--Zelevinsky presentation of $\pi_{x_0}$ are linked. Then there is a function $\mc{E}\in R$, and an open and closed neighborhood $X_0$ of $x_0$ in $X$, such that for any $x\in\Sigma\cap X_0$, we have
\[\mc{E}(x)=\epsilon(\pi_x,s,\psi,\varrho).\]
\end{introthm}

The presence of the representation $\varrho$ in the statement of the theorem is so that it applies to all functorial lifts of the members of the family $\mc{F}$; see Definition \ref{defepsilon} for the precise definition of the corresponding epsilon factor. When $\varrho$ is the standard representation of $GL_n(\C)$, this is just the usual epsilon factor $\epsilon(\pi_x,s,\psi)$. In general, we prove this theorem for the epsilon factors attached to Rankin--Selberg products for arbitrarily many families, where the families are themselves allowed to be families of representations of products of several groups $GL_{n_i}(F_i)$. This can be used to give local constancy statements about, say, signs of functional equations of triple product $L$-functions; see Section \ref{secremeigen}.

So, at least under a particular ``unlinkedness'' hypothesis, epsilon factors locally vary coherently in families. This unlinkedness hypothesis is often satisfied for families coming from eigenvarieties for reasons having to do with temperedness; see again Section \ref{secremeigen}, and also Proposition \ref{proptempunl}.

The proof of Theorem \ref{ithmepsilon} is significantly more involved than the proofs of Theorems \ref{ithmfamilythm} or \ref{ithmK1fixed}, and most of the technical work takes place in Section \ref{secdetect}. The idea is as follows. Since $\pi_{x_0}$ satisfies the unlinkedness hypothesis in the statement of the theorem, by Theorem \ref{ithmfamilythm}, it can be written as an irreducible parabolic induction
\[\pi_{x_0}\cong\Ind_{P(F)}^{GL_{n}(F)}(Q(\Delta_{1})\boxtimes\dotsb\boxtimes Q(\Delta_{r}))\]
for some segments $\Delta_{k}$, where $P$ is the appropriate standard parabolic subgroup of $GL_{n}$. From Proposition \ref{propunlinked}, it follows that there is an open and closed neighborhood of $x_0$, which we can take to be our $X_0$, such that for any $x\in\Sigma\cap X_0$, we have
\begin{equation*}
\pi_{x}\cong\Ind_{P(F)}^{GL_n(F)}(Q(\Delta_{1}\otimes\chi_{1,x})\boxtimes\dotsb\boxtimes Q(\Delta_{r}\otimes\chi_{r,x}))
\end{equation*}
for some unramified characters $\chi_{k,x}$ of $F^\times$.

It will suffice to construct the epsilon factor of an induced representation like the ones above in a uniform way as a polynomial combination of traces of certain Hecke operators. Let $\varpi$ be a uniformizer for $F$. Using formulas for the behavior of epsilon factors under twists (Lemma \ref{lemepstwists}), we will reduce this problem to constructing certain products of the numbers $\chi_{k,x}(\varpi)$ as polynomial combinations of traces of certain Hecke operators. By the theory of symmetric polynomials, it suffices to construct instead certain sums of powers of the numbers $\chi_{k,x}(\varpi)$ using traces of Hecke operators.

For any of the segments $\Delta_{k}\otimes\chi_{k,x}$ which consist of unramified characters, this task is easy to accomplish for the individual representation $Q(\Delta_{k}\otimes\chi_{k,x})$; indeed, if $\ell$ is the length of this segment, if $I$ is the standard Iwahori subgroup of $GL_\ell(F)$, if $a$ is a positive integer and $f_a$ is the characteristic function of the double coset $I\diag(\varpi^a,1,\dotsc,1)I$, then it is not too hard to check that $\tr(f_a|Q(\Delta_{k}\otimes\chi_{k,x}))$ is a fixed (that is, independent of $x$) nonzero constant multiple of $\chi_{k,x}(\varpi)^a$.

Of course there are two problems to making this approach work in general, namely that $\Delta_{k}$ of course contains nontrivial supercuspidals in general, and that we only have access to the traces of Hecke operators on the full induced representation which we have used to represent $\pi_x$ above. This second problem is where the theory of covers plays a crucial role. This theory allows us to transfer Hecke operators acting on the Jacquet modules of certain induced representations to Hecke operators acting on the full induced representations themselves, and it does so in a way which preserves traces. Since the Jacquet modules of induced representations can be described explicitly (as we do in Proposition \ref{propcasselman} and Lemma \ref{lemsimplejac}), it suffices to construct our desired power sums of $\chi_{k,x}(\varpi)$ as traces of Hecke operators on certain Jacquet modules. Since Jacquet modules themselves are in general not irreducible, we will need to know that the combinations of traces we get on each constituent can themselves be recombined to get the desired power sums; this is where we use Lemma \ref{lemdeterminant}.

This reduces the problem, more or less, to the problem of handling supercuspidal representations, and this is accomplished by using the theory of types and their associated Hecke algebras, especially Lemma \ref{lemtwindideal}.

With the sketch of the proof of Theorem \ref{ithmepsilon} now concluded, we now make a few remarks. We would like to mention that Disegni has proved a similar result on the variation of local epsilon factors for families of a different nature which arise in the context of other eigenvariety constructions; see \cite{disegni}. The families that appear there are defined using the theory of co-Whittaker modules and are of the type considered in the work of Emerton and Helm. It would be interesting to be able to compare that notion of family with the one that is considered in this paper.

Also, an earlier version of this paper proved Theorem \ref{thmepsilon} only for epsilon factors for just one family and attached to just the standard representation of the dual $GL_n(\C)$. We thank an anonymous referee for pointing out that this can often be deduced from the existence of universal gamma factors shown in the work of G. Moss, see \cite{Mossgamma, Mossloc}. We note that in order to make this deduction, one needs coherence properties of the corresponding $L$-factors, and if the family has points whose supercuspidal support contains any unramified characters, this may not be obvious; this is still the case even in the context of Theorem \ref{thmfamilythm}, since this theorem alone has no implications towards the coherence of the unramified characters $\chi_{k,i,x}$ in its statement.

Finally, the contents of this paper are divided into seven sections. Section \ref{secBZ} reviews the Bernstein--Zelevinsky classification. Sections \ref{secK1}, \ref{secfamilies} and \ref{secepsilon} are devoted to the proofs of Theorems \ref{ithmK1fixed}, \ref{ithmfamilythm} and \ref{ithmepsilon}, respectively (or rather, their more general versions given in the text). Section \ref{secdetect} has the main technical preparation for the proof of Theorem \ref{ithmepsilon}, and section \ref{secunlinked} contains results about unlinked families. Finally, section \ref{secremeigen} contains remarks about the relation of the families considered here to those coming from the theory of eigenvarieties.

\subsection*{Acknowledgements}
I would like to thank Marco Sangiovanni Vincentelli for reading a first draft of this paper and giving suggestions on the exposition. I also thank Eric Chen and Chris Skinner for helpful conversations. This material is based upon work supported by the National Science Foundation under Award DMS-2102473.

\subsection*{Notation and conventions}
Given a nonarchimedean local field $F$ of characteristic zero, we always write $\O_F$ for its ring of integers.

Given a standard parabolic subgroup $P$ of $GL_n$ with standard Levi $M$ and unipotent radical $N$, and a smooth admissible representation $\tau$ of $M(F)$, we always write
\[\Ind_{P(F)}^{GL_n(F)}(\tau)\]
for the unitary induction of $\tau$; thus $\tau$ is extended trivially to $N(F)$ and twisted by the square root of the modulus character of $P(F)$ before inducing.

We will have occasion in Section \ref{secK1} to consider parabolically induced representations of finite general linear groups; again these are induced by extending the inducing representation trivially over the unipotent radical, but no twist is added. The same symbol $\Ind$ will be used in this case as well.

Given several general linear groups $GL_{n_1}(F_1),\dotsc,GL_{n_r}(F_r)$ over nonarchimedean local fields $F_1,\dotsc,F_r$ of characteristic zero, and smooth admissible representations $\pi_i$ of each $GL_{n_i}(F_i)$, we always write
\[\pi_1\boxtimes\dotsb\boxtimes\pi_r\]
for the exterior tensor product, considered as a representation of $GL_{n_1}(F)\times\dotsb\times GL_{n_r}(F)$.

\section{Background on the Bernstein--Zelevinsky classification}
\label{secBZ}

Fix a local nonarchimedean field $F$ of characteristic zero. We recall here the classification of irreducible smooth representations of $GL_n(F)$, $n\geq 1$, in terms of supercuspidal ones, due to Bernstein and Zelevinsky \cite{zel}. The main results of this paper are phrased in terms of this classification.

First, for any positive integer $m$, a finite collection of supercuspidal representations of $GL_m(F)$ of the form
\[\Delta=\{\tau,\tau\otimes\vert\det\vert,\dotsc,\tau\otimes\vert\det\vert^{\ell-1}\},\]
where $\ell$ is a positive integer, will be called a \textit{segment}. The integer $\ell$ is called the \textit{length} of the segment $\Delta$. Given such a segment $\Delta$, if we let $P\subset GL_{m\ell}$ be the standard parabolic subgroup with Levi $M=(GL_m)^{\times\ell}$, then we can form the parabolically induced representation
\[\Ind_{P(F)}^{GL_{m\ell}(F)}(\tau\boxtimes\dotsb\boxtimes(\tau\otimes\vert\det\vert^{\ell-1}))\]
(Recall our convention that this induction is normalized by the the square root of the modulus character of $P(F)$.) This induced representation has a unique irreducible quotient which we will denote by $Q(\Delta)$.

We will consider finite multisets (that is, finite sets whose elements have finite multiplicities) of segments, say $S=\{\Delta_1,\dotsc,\Delta_r\}$, where the $\Delta_i$'s, $1\leq i\leq r$ are (not necessarily distinct) segments, each consisting supercuspidal representations of $GL_{m_i}(F)$ for possibly different integers $m_i\geq 1$. In fact, for $1\leq i\leq r$, let $\ell_i$ be the length of $\Delta_i$. Let $P$ now be the standard parabolic subgroup of $GL_n$, $n=m_1\ell_1+\dotsb+m_r\ell_r$, with Levi subgroup $GL_{m_1\ell_1}\times\dotsb\times GL_{n_r\ell_r}$. Then according to \cite{zel}, if the segments $\Delta_1,\dotsc,\Delta_r$ are ordered in a particular way (which we specify just below) then the induced representation
\[\pi(S)=\pi(\Delta_1,\dotsc,\Delta_r)=\Ind_{P(F)}^{GL_n(F)}(Q(\Delta_1)\boxtimes\dotsb\boxtimes Q(\Delta_r))\]
has again a unique irreducible quotient, which we denote by $Q(\Delta_1,\dotsc,\Delta_r)$ or $Q(S)$. The main result of \textit{loc. cit.} is that any irreducible smooth representation can be written in this way, and moreover the (unordered) multiset $S$ is determined from $Q(S)$.

We call the multiset consisting of the supercuspidal representations contained in the $\Delta_i$'s up to unramified twists the \textit{supercuspidal support} of $Q(S)$, or of $S$. Thus each member of a supercuspidal support is an equivalence class of supercuspidal representations of some $GL_n(F)$ where two such representations are considered equivalent if one is obtained as an unramified twist of the other.

The order in which one must put the elements of $S$ to form $Q(S)$ is described as follows. First, given two segments $\Delta,\Delta'$, we say that $\Delta$ and $\Delta'$ are \textit{linked} if neither $\Delta$ nor $\Delta'$ is contained in the other and $\Delta\cup\Delta'$ is an interval. (This condition is empty if $\Delta$ and $\Delta'$ consist of representations of $GL_m(F)$ each for different integers $m$). Say the first element of $\Delta$ is $\tau$ and that of $\Delta'$ is $\tau'$. Then we say $\Delta$ \textit{precedes} $\Delta'$ if $\Delta$ and $\Delta'$ are linked and $\tau'=\tau\otimes\vert\det\vert^{i}$ for some integer $i\geq 1$. Then, in forming the induced representation above, we must require that the multiset $S$ is ordered so that for $1\leq i<j\leq r$, we have that $\Delta_i$ does not precede $\Delta_j$. It is always possible to order $S$ in this way.

Finally, the following partial order on multisets $S$ of segments will play a role in this paper. Let $S=\{\Delta_1,\dotsc,\Delta_r\}$ be a multiset of segments. First, assume the two segments $\Delta_i,\Delta_j$, for some $i$ and $j$ with $i\ne j$ and $1\leq i<j\leq r$, are linked. Then we may consider the segment
\[S'=\{\Delta_1,\dotsc,\widehat{\Delta}_i\dotsc,\widehat{\Delta}_j\dotsc,\Delta_r,(\Delta_i\cup\Delta_j),(\Delta_i\cap\Delta_j)\},\]
where the symbols $\widehat{\Delta}_i$ and $\widehat{\Delta}_j$ mean that $\Delta_i$ and $\Delta_j$ are omitted from $S'$, and where if $\Delta_i\cap\Delta_j$ is empty, then we omit it from $S'$ as well. We then say the segment $S'$ is \textit{obtained from} $S$ \textit{by an elementary operation}. If $S_0$ is another multiset of segments, we write $S_0<S$ if there are multisets of segments $S_1,\dotsc,S_q=S$, for some $q\geq 1$, such that each $S_{i-1}$, $1\leq i\leq q$, is obtained from $S_i$ by an elementary operation. We also write $S_0\leq S$ if $S_0<S$ or $S_0=S$. Then the relation $\leq$ defines a partial ordering on multisets of segments.

\section{A Theorem on fixed vectors by a certain compact open subgroup}
\label{secK1}

We retain the notation of the previous section. In particular we have our nonarchimedean local field $F$ of characteristic zero. Fix a uniformizer $\varpi$ in $\O_F$.

For any positive integer $n$, we consider the compact open subgroup $K_{1,n}$ of $GL_n(F)$ defined by
\[K_{1,n}=\sset{g\in GL_n(\O_F)}{g\equiv 1\modulo{\varpi}}.\]
Then $K_{1,n}$ is normal in $GL_n(\O_F)$, and $GL_n(\O_F)/K_{1,n}$ may be naturally identified with $GL_n(\O_F/\varpi)$. As such, given a smooth admissible representation $\pi$ of $GL_n(F)$, the space $\pi^{K_{1,n}}$ of $K_{1,n}$-fixed vectors in $\pi$ is naturally a finite dimensional (by admissibility) representation of the group $GL_n(\O_F/\varpi)$.

Let $P=MN$ be the standard parabolic subgroup of $GL_n$ with Levi $M=GL_{n_1}\times\dotsb\times GL_{n_r}$, $n_1+\dotsb +n_r=n$, and unipotent radical $N$. Then
\[K_{1,n}\cap M=K_{1,n_1}\times\dotsb\times K_{1,n_r}.\]
We have the following proposition relating two parabolic inductions.

\begin{proposition}
\label{propindfixed}
Let $P=MN$ be as above, and let $\tau$ be an irreducible smooth representation of $M(F)$. Then there is an isomorphism of $GL_n(\O_F/\varpi)$-representations
\[\Ind_{P(F)}^{GL_n(F)}(\tau)^{K_{1,n}}\cong\Ind_{P(\O_F/\varpi)}^{GL_n(\O_F/\varpi)}(\tau^{K_{1,n}\cap M(F)}).\]
\end{proposition}

The proof will require the following easy lemma.

\begin{lemma}
\label{lemindfixed}
Let $P=MN$ be as above. Let $K\subset GL_n(\O_F)$ be a normal open subgroup, and let $\tau$ be a smooth admissible representation of $M(F)$. Let $g_1,\dotsc,g_m\in GL_n(\O_F)$ be a set of coset representatives for $P(F)\backslash GL_n(F)/K$ (which exist by the Iwasawa decomposition), and let $v_1,\dotsc,v_k$ be a basis of $\tau^{K\cap P(F)}$ as a $\C$-vector space. Then for any $1\leq i\leq m$ and $1\leq l\leq k$, the space
\[\Ind_{P(F)}^{GL_n(F)}(\tau)^{K},\]
contains a unique function $f_{i,l}$ such that $f_{i,l}(g_j)=\delta_{i,j}v_l$ (Kronecker delta), and such functions form a basis of this space.
\end{lemma}

\begin{proof}
Because of the Iwasawa decomposition $GL_n(F)=P(F)GL_n(\O_F)$, any function
\[f\in\Ind_{P(F)}^{GL_n(F)}(\tau)^{K}\]
is determined by its values on the $g_i$'s. Moreover, such functions must take values in $\tau^{K\cap P(F)}$. Thus the functions $f_{i,l}$ form a basis as long as they are well defined.

To see that the $f_{i,l}$'s are indeed well defined, let $k,k'\in GL_n(\O_F)$ and $p,p'\in P(F)$ be such that $pk=p'k'$. Say $kK=g_jK$ and $k'K=g_{j'}K$. Then clearly $j=j'$, whence $k_j^{-1}p^{-1}pk_j\in K$. Thus $p^{-1}p'\in K$ by the normality of $K$ in $GL_n(\O_F)$, and so $p(K\cap P(F))=p'(K\cap P(F))$. Hence
\[pf_{i,l}(k_j)=p'f_{i,l}(k_{j'}),\]
so the $f_{i,l}$'s are well defined.
\end{proof}

\begin{proof}[Proof (of Proposition \ref{propindfixed}).]
Since clearly we have $K_{1,n}\cap P(F)=(K_{1,n}\cap M(F))(K_{1,n}\cap N(F))$, we get that $\tau^{K_{1,n}\cap P(F)}=\tau^{K_{1,n}\cap M(F)}$. Let $v_1,\dotsc,v_k$ be a basis of $\tau^{K_{1,n}\cap M(F)}$.

We now apply the lemma with $K=K_{1,n}$; consequently, $\Ind_{P(F)}^{GL_n(F)}(\tau)^{K_{1,n}}$ is in bijection with the space of functions $f:P(\O_F)\backslash GL_n(\O_F)/K_{1,n}\to\tau^{K_{1,n}\cap M(F)}$, or what is the same, $\Ind_{P(\O_F/\varpi)}^{GL_n(\O_F/\varpi)}(\tau^{K_{1,n}\cap M(F)})$. Since the actions of $GL_n(\O_F/\varpi)$ on these spaces are defined by right translation, this gives the desired isomorphism of representations.
\end{proof}

We will also require the following lemma for the next section.

\begin{lemma}
\label{lemunrtwist}
Let $P=MN$ be as above and let $K\subset GL_n(\O_F)$ be a normal open subgroup. Let $\tau$ be a smooth admissible representation of $M(F)$ and let $\chi$ be an unramified character of $M(F)$. Then the representations of $GL_n(\O_F)/K$
\[\Ind_{P(F)}^{GL_n(F)}(\tau)^{K}\qquad\textrm{and}\qquad\Ind_{P(F)}^{GL_n(F)}(\tau\otimes\chi)^{K}\]
are isomorphic.
\end{lemma}

\begin{proof}
Use Lemma \ref{lemindfixed} to write down a basis of both induced representations and note that the restriction to $K$ of the action on either is the same by the unramifiedness of $\chi$.
\end{proof}

Recall that the finite group $GL_n(\O_F/\varpi)$ has a finite dimensional Steinberg representation $\St_{n,\varpi}$ (see \cite{hum} for a detailed discussion). It is a genuine representation but it may be defined virtually by the character formula
\[\St_{n,\varpi}=\sum_{P\subset GL_n}(-1)^{\vert P\vert}\Ind_{P(\O_F/\varpi)}^{GL_n(\O_F/\varpi)}(1_{P(\O_F/\varpi)}),\]
where the sum is over all $2^{n-1}$ standard parabolic subgroups $P=MN$ of $GL_n$, where $\vert P\vert$ denotes the number of simple roots in the Levi $M$ of $P$, and where $1_{P(\O_F/\varpi)}$ denotes the trivial representation of $P(\O_F/\varpi)$.

\begin{proposition}
\label{propfinst}
Let $\chi$ be an unramified character of $GL_n(F)$. Then as representations of $GL_n(\O_F/\varpi)$, we have
\[(\St_n\otimes\chi)^{K_{1,n}}\cong\St_{n,\varpi},\]
where $\St_n$ denotes the usual Steinberg representation of $GL_n(F)$.
\end{proposition}

\begin{proof}
The character formula of \cite[Proposition 9.13]{zel}, applied with the segments
\[\{\chi\vert\cdot\vert^{(1-n)/2}\},\dotsc,\{\chi\vert\cdot\vert^{(n-1)/2}\}\]
of length $1$ implies that, virtually,
\[1_{GL_n(F)}=\sum_{P}(-1)^{\vert P\vert}\iota_{P(F)}^{GL_n(F)}(\tau(P)),\]
where the sum is over all standard parabolic subgroups $P=MN$ of $GL_n$, and $\tau(P)$ is the unique irreducible quotient of
\[\iota_{(B\cap M)(F)}^{M(F)}(\chi\vert\cdot\vert^{(1-n)/2}\boxtimes\dotsb\boxtimes\chi\vert\cdot\vert^{(n-1)/2});\]
this representation $\tau(P)$ is given by an unramified twist of the Steinberg on each block of the Levi $M(F)$. Here, $B$ denotes the standard Borel subgroup of $GL_n$.

We then apply Zelevinsky's involution $(\cdot)^t$ (cf. \cite[\S 9]{zel}) to this identity, which we can do because he proves it is an algebra involution of the Grothendieck ring of smooth admissible representations of $GL_n(F)$, $n\geq 1$. Since this involution switches the Steinberg representation with the trivial representation, this gives
\[\St_n=\sum_{P}(-1)^{\vert P\vert}\iota_{P(F)}^{GL_n(F)}(1_{P(F)}).\]
Taking $K_{1,n}$ fixed vectors (which is exact) and applying Proposition \ref{propindfixed} yields the proposition at hand.
\end{proof}

\begin{corollary}
\label{corstsize}
We have
\[\dim_{\C}(\St_n^{K_{1,n}})=\vert\O_F/\varpi\vert^{n(n-1)/2}.\]
\end{corollary}

\begin{proof}
This follows from the formula
\[\vert\O_F/\varpi\vert^{n(n-1)/2}=\sum_{P}(-1)^{\vert P\vert}\vert (P\backslash GL_n)(\O_F/\varpi)\vert=\dim_{\C}(\St_{n,\varpi})\]
along with the proposition.
\end{proof}

We now prove the main result of this section.

\begin{theorem}
\label{thmK1fixed}
Let $\pi$ be an irreducible smooth representation of $GL_n(F)$ whose supercuspidal support consists of unramified characters. Write $\pi=Q(\Delta_1,\dotsc,\Delta_r)$ where the segments $\Delta_i$ consist of unramified characters. Let $\ell_i$ be the length of the segment $\Delta_i$. Then
\[\frac{v_p(\dim_{\C}(\pi^{K_{1,n}}))}{v_p(\vert\O_F/\varpi\vert)}=\sum_{i=1}^r\frac{\ell_i(\ell_i-1)}{2}.\]
\end{theorem}

\begin{proof}
Let $P$ be the standard parabolic subgroup of $GL_n$ corresponding to the partition $n=\ell_1+\dotsb+\ell_r$. Then by Propositions \ref{propindfixed} and \ref{propfinst}, we have
\[\pi(\Delta_1,\dotsc,\Delta_r)^{K_{1,n}}\cong\Ind_{P(\mc{O}/\varpi)}^{GL_n(\O_F/\varpi)}(\St_{\ell_1,\varpi}\boxtimes\dotsb\boxtimes\St_{\ell_r,\varpi}).\]
By Corollary \ref{corstsize}, the right hand side has dimension
\[\vert(P\backslash GL_n)(\O_F/\varpi)\vert\prod_{i=1}^r\vert\O_F/\varpi\vert^{\ell_i(\ell_i-1)/2}.\]
Therefore, since $p$ does not divide the order of any of the finite flag varieties $(P\backslash GL_n)(\O_F/\varpi)$ (this follows from standard formulas for the number of points of such flag varieties over finite fields) we have
\[\frac{v_p(\dim_{\C}\pi(\Delta_1,\dotsc,\Delta_r)^{K_{1,n}})}{v_p(\vert\O_F/\varpi\vert)}=\sum_{i=1}^r\frac{\ell_i(\ell_i-1)}{2}.\]

Now assume that two of the segments $\Delta_j,\Delta_k$, $j\ne k$, are linked. Assume without loss of generality that $\ell_j\geq \ell_k$. Then consider the representation
\[\sigma=\pi(\Delta_1,\dotsc,\widehat{\Delta}_j,\dotsc,\widehat{\Delta}_k,\dotsc,\Delta_r,\Delta_i\cup\Delta_j,\Delta_i\cap\Delta_j),\]
where $\widehat{\Delta}_i$ means $\Delta_i$ is omitted. By the same reasoning as above, we have
\begin{equation*}
\frac{v_p(\dim_{\C}(\sigma^{K_{1,n}}))}{v_p(\vert\O_F/\varpi\vert)}=\frac{(\ell_j+\ell_k)(\ell_j+\ell_k-1)}{2}+\frac{(\ell_j-\ell_k)(\ell_j-\ell_k-1)}{2}+\sum_{\substack{i=1\\ i\ne j,k}}^r\frac{\ell_i(\ell_i-1)}{2}.
\end{equation*}
Therefore
\begin{multline*}
\frac{v_p(\dim_{\C}(\sigma^{K_{1,n}}))}{v_p(\vert\O_F/\varpi\vert)}-\frac{v_p(\dim_{\C}\pi(\Delta_1,\dotsc,\Delta_r)^{K_{1,n}})}{v_p(\vert\O_F/\varpi\vert)}\\
\begin{aligned}
&=\frac{(\ell_j+\ell_k)(\ell_j+\ell_k-1)}{2}+\frac{(\ell_j-\ell_k)(\ell_j-\ell_k-1)}{2}-\frac{\ell_j(\ell_j-1)}{2}-\frac{\ell_k(\ell_k-1)}{2}\\
&=\ell_k>0.
\end{aligned}
\end{multline*}
It follows that if $S$ is any multiset of segments with $S<\{\Delta_1,\dotsc,\Delta_r\}$ (strict inequality) then we have
\begin{equation}
\label{eqdiffdims}
v_p(\dim_{\C}(\pi(S)^{K_{1,n}}))>v_p(\dim_{\C}\pi(\Delta_1,\dotsc,\Delta_r)^{K_{1,n}}).
\end{equation}

Now we claim that
\[v_p(\dim_{\C}(\pi(\Delta_1,\dotsc,\Delta_r)^{K_{1,n}}))=v_p(\dim_{\C}(\pi^{K_{1,n}})).\]
By the considerations above, this will finish the proof. To see the claim, we induct on the number $N$ of multisets $S$ of segments with $S<\{\Delta_1,\dotsc,\Delta_r\}$. If $N=0$ then $\pi(\Delta_1,\dotsc,\Delta_r)$ is irreducible, hence equal to $\pi$. If $N>0$, then we invoke \cite[Theorem 7.1]{zel}, which says that virtually,
\[\pi(\Delta_1,\dotsc,\Delta_r)=\pi+\sum_{S<\{\Delta_1,\dotsc,\Delta_r\}}m(S)Q(S),\]
for some positive integers $m(S)$. Thus
\[\dim_{\C}(\pi(\Delta_1,\dotsc,\Delta_r)^{K_{1,n}})=\dim_{\C}(\pi^{K_{1,n}})+\sum_{S<\{\Delta_1,\dotsc,\Delta_r\}}m(S)\dim_{\C}(Q(S)^{K_{1,n}}).\]
By the induction hypothesis and \eqref{eqdiffdims}, we have
\[v_p(\dim_{\C}(Q(S)^{K_{1,n}}))<v_p(\dim_{\C}(\pi(\Delta_1,\dotsc,\Delta_r)^{K_{1,n}})),\]
for all $S<\{\Delta_1,\dotsc,\Delta_r\}$, which, by the strict triangle inequality, forces
\[v_p(\dim_{\C}(\pi(\Delta_1,\dotsc,\Delta_r)^{K_{1,n}}))=v_p(\dim_{\C}(\pi^{K_{1,n}})),\]
as desired. This finishes the proof.
\end{proof}

We record the following corollary of the proof which will be useful later.

\begin{corollary}
\label{corpfineq}
Let $S=\{\Delta_1,\dotsc,\Delta_r\}$ and $S'=\{\Delta_1',\dotsc,\Delta_{r'}'\}$ be multisets of segments. For $i$ with $1\leq i\leq r$ (resp. $j$ with $1\leq j\leq r'$), let $\ell_i$ (resp. $\ell_j'$) be the length of $\Delta_i$ (resp. $\Delta_j'$). Then if $S'>S$, then
\[\sum_{j=1}^{r'}\frac{\ell_j'(\ell_j'-1)}{2}>\sum_{i=1}^r\frac{\ell_i(\ell_i-1)}{2}.\]
\end{corollary}

Recall that the local Langlands correspondence of Harris--Taylor and Henniart attaches a Weil--Deligne representation $(\rho,N)$ to any irreducible smooth representation $\pi$ of $GL_n(F)$; here $\rho:W_F\to GL(V_\rho)$ is a continuous representation of the Weil group of $F$ on an $n$ dimensional $\C$-vector space $V_\rho$ and $N$ is a nilpotent ``monodromy" operator on $V_\rho$, both satisfying a certain compatibility relation. This assignment has the property that $\St_n$ is taken to a certain representation $\Sp(n)=(1_n,N_{n})$, called the \textit{special representation} of dimension $n$, whose space $V$ has a basis $v_1,\dotsc,v_n$ such that $N_nv_i=v_{i+1}$ for $i=1,\dotsc,n-1$ and $N_nv_n=0$, and where the representation $1_n$ is the trivial representation. Moreover, if $\pi$ is supercuspidal, then the associated representation $\rho$ is irreducible when restricted to the inertia group $I_F$ and the associated operator $N$ is zero.

If $\Delta=\{\tau,\dotsc,\tau\otimes\vert\det\vert^{\ell-1}\}$ is a segment with $\tau$ a supercuspidal representation of $GL_n(F)$ with associated Weil--Deligne representation $(\rho,N)$, then associated with $Q(\Delta)$ is the representation $(\rho\otimes 1_n,\id\otimes N_n)$. Further, if $\Delta_1,\dotsc,\Delta_r$ are segments, then $Q(\Delta_1,\dotsc,\Delta_r)$ has associated representation 
\[(\rho_1\oplus\dotsb\oplus\rho_r,N_1\oplus\dotsc\oplus N_r),\]
where $(\rho_i,N_i)$ is the Weil--Deligne representation attached to $Q(\Delta_i)$.

The following is thus way to rephrase the above theorem in terms of Weil--Deligne representations.

\begin{corollary}
\label{cormonodcount}
Let $\pi$ be an irreducible smooth representation of $GL_n(F)$ with associated Weil--Deligne representation $(\rho,N)$. Assume $\rho|_{I_F}$ is trivial. Let $v_1,\dotsc,v_n$ be a basis of the space $V$ of $(\rho,N)$ in which $N$ is written in Jordan form. Then
\[\frac{v_p(\dim_{\C}(\pi^{K_{1,n}}))}{v_p(\vert\O_F/\varpi\vert)}\]
is the number of nonzero entries of the matrix of $\exp(N)-\id_V$ in the basis $v_1,\dotsc,v_n$.
\end{corollary}

\section{Application to families of admissible representations}
\label{secfamilies}

We now consider the variation of smooth admissible representations in families; these families are defined to parametrize Hecke traces and they arise in the context of eigenvarieties.

Fix throughout this section a finite collection of (not necessarily distinct) nonarchimedean local fields of characteristic zero $F_1,\dotsc,F_N$ and positive integers $n_1,\dotsc,n_N$. For each $i=1,\dotsc,N$, we will consider the Hecke algebras $\mc{H}_i=C_c^\infty(GL_{n_i}(F_i),\C)$. Then given a smooth admissible representation $\pi_i$ of $GL_{n_i}(F_i)$, the algebra $\mc{H}_i$ acts on $\pi_i$ by convolution, and given any $f_i\in\mc{H}_i$, we can consider the trace $\tr(f_i|\pi_i)\in\C$. The representation $\pi_i$ is determined by the trace map $\tr(\cdot|\pi_i)$.

We also consider the Hecke algebra $\mc{H}=C_c^\infty(\prod_i GL_n(F_i),\C)$, so that
\[\mc{H}\cong\mc{H}_1\otimes_\C\dotsb\otimes_\C\mc{H}_N.\]
Any irreducible smooth representation $\pi$ of $\prod_i GL_{n_i}(F_i)$ then decomposes as an exterior tensor product,
\[\pi\cong\pi_1\boxtimes\dotsb\boxtimes\pi_N,\]
where each $\pi_i$ is a smooth admissible representation of $GL_{n_i}(F_i)$, $i=1,\dotsc,N$.

\begin{definition}
\label{deffamily}
A \textit{family of smooth admissible representations of} $\prod_i GL_{n_i}(F)$ is a quadruple $\mc{F}=(X,\Sigma,R,T)$ where:
\begin{itemize}
\item $X$ is a topological space,
\item $\Sigma$ is a dense subset of $X$,
\item $R$ is a $\C$-subalgebra of the ring of $\C$-valued functions on $X$,
\item $T$ is a $\C$-linear map $\mc{H}\to R$,
\end{itemize}
which satisfies the following properties:
\begin{enumerate}[label=(\roman*)]
\item For any $a\in\C$ and $\phi\in R$, the set $\phi^{-1}(a)\subset X$ is closed in $X$,
\item For any $x\in\Sigma$, there exists a smooth admissible representation $\pi_x$ of $\prod_i GL_{n_i}(F_i)$ such that for any $f\in\mc{H}$, the composition $T_x=\ev_x\circ T$ of $T$ with the evaluation map $\ev_x:R\to\C$ of functions in $R$ at $x$ satisfies $T_x(f)=\tr(f|\pi_x)$.
\end{enumerate}
Moreover, we call $\mc{F}$ \textit{generically irreducible} if for all $x\in\Sigma$, the representations $\pi_x$ of condition (ii) above can be taken to be irreducible.
\end{definition}

The main result of this section is the following.

\begin{theorem}
\label{thmfamilythm}
Let $\mc{F}=(X,\Sigma,R,T)$ be a generically irreducible family of smooth admissible representations of $\prod_i GL_{n_i}(F_i)$. For $x\in\Sigma$, let $\pi_{i,x}$ be irreducible smooth representations of $GL_{n_i}(F_i)$, $i=1,\dotsc,N$, such that
\[T_x(f)=\tr(f|\pi_{1,x}\boxtimes\dotsb\boxtimes\pi_{N,x}).\]
Fix $x_0\in\Sigma$ and write
\[\pi_{i,x_0}\cong Q(\Delta_{1,i},\dotsc,\Delta_{r,i})\]
in the Bernstein--Zelevinsky classification, for some segments $\Delta_{1,i},\dotsc,\Delta_{r,i}$. Then there is an open and closed neighborhood $X_0$ of $x_0$ in $X$ with the following property: For any $x\in \Sigma\cap X_0$, there are unramified characters $\chi_{k,i,x}$ of $F_i^\times$, $1\leq k\leq r$,  $1\leq i\leq N$, such that
\[\pi_{i,x}\cong Q(\Delta_{1,i}\otimes\chi_{1,i,x},\dotsc,\Delta_{r,i}\otimes\chi_{r,i,x}).\]
Here, for any $k$ with $1\leq k\leq r_0$, the segment $\Delta_{k,i}\otimes\chi_{k,i,x}$ denotes the segment whose members are just those of $\Delta_{k,i}$ each twisted by $\chi_{k,i,x}\circ\det$.
\end{theorem}

\begin{remark}
In terms of Weil--Deligne representations, the theorem says the following. For any $x\in\Sigma\cap X_0$, let $(\rho_{i,x},N_{i,x})$ be the Weil--Deligne representation attached to $\pi_{i,x}$, $1\leq i\leq N$. Then the pairs $(\rho_{i,x}|_{I_{F_i}},N_{i,x})$ are constant on $\Sigma\cap X_0$; only the Frobenius action changes.
\end{remark}

The proof of Theorem \ref{thmfamilythm} has three main ingredients: There is the paper of Schneider and Zink \cite{SZ} where they construct certain types in the sense of Bushnell--Kutzko, there is the Arthur--Clozel \cite{AC} local base change for $GL_n$, and there are the results of the previous section.

\begin{proof}[Proof (of Theorem \ref{thmfamilythm})]
The proof will proceed in two separate steps. In Step 1 we find an open and closed subset $X_0\subset X$ such that for every $x\in \Sigma\cap X_0$, we have $\pi_{i,x}\cong Q(S_{i,x})$ for some multiset of segments $S_{i,x}$ with $S_{i,x}\geq\{\Delta_{1,i}\otimes\chi_{1,i,x},\dotsc,\Delta_{r,i}\otimes\chi_{r,i,x}\}$ for some unramified characters $\chi_{1,i,x},\dotsc,\chi_{r,i,x}$ of $F_i^\times$. In Step 2, we show the reverse inequality (after possibly shrinking $X_0$).

\textit{Step 1.} First, for any $i$ with $1\leq i\leq N$, \cite[Proposition 6.2 i]{SZ} implies the existence of an idempotent $e_i\in\mc{H}_i$ such that
\[e_i Q(\Delta_{1,i}\otimes\chi_{1,i},\dotsc,\Delta_{r,i}\otimes\chi_{r,i})\ne 0,\]
for any unramified characters $\chi_{1,i},\dotsc,\chi_{r,i}$ of $F_i^\times$. Moreover, if $S$ is a multiset of segments with the same supercuspidal support as $\{\Delta_{1,i},\dotsc,\Delta_{r,i}\}$, then \cite[Proposition 6.2 ii]{SZ} implies that if $e_i Q(S)\ne 0$, then there are unramified characters $\chi_{1,i},\dotsc,\chi_{r,i}$ of $F_i^\times$ such that $S\geq\{\Delta_{1,i}\otimes\chi_{1,i},\dotsc,\Delta_{r,i}\otimes\chi_{r,i}\}$. Finally, if instead $S$ is a multiset of segments with supercuspidal support which is different from that of $\{\Delta_{1,i},\dotsc,\Delta_{r,i}\}$ even up to unramified twist, then $e_iQ(S)=0$. The idempotent $e_i$ actually comes from a type, so there is an open compact subgroup $K_i\subset GL_{n_i}(\O_{F_i})$ and a smooth representation $\rho_i$ of $K_i$ such that $e_i$ is given by the formula
\[e_i(g)=\begin{cases}
\frac{\dim\rho_i}{\vol(K_i)}\tr(g^{-1}|\rho_i)&\textrm{if }g\in K_i;\\
0&\textrm{otherwise.}
\end{cases}\]

Now say $S$ and $S'$ are multisets of segments with the same supercuspidal support as $\{\Delta_{1,i},\dotsc,\Delta_{r,i}\}$ up to unramified twist, and that $S'$ is obtained from $S$ by twisting each of its segments by an unramified character. Let $K_i'\subset\ker(\rho_i)$ be an open subgroup which is normal in $GL_{n_i}(F_i)$. Then by Lemma \ref{lemunrtwist} applied to $\pi(S)^{K_i'}$ and $\pi(S')^{K_i'}$, the spaces $e_i\pi(S)$ and $e_i\pi(S')$ give the same representations of $K_i$. It follows that, since there are only finitely many such multisets $S$ up to unramified twist, there can only be finitely many possible values for the quantities
\[\tr(e_i|Q(S))=\dim_\C(e_iQ(S)).\]

Now consider the function $\phi=T(e_1\otimes\dotsb\otimes e_N)\in R$. Then $\phi(x_0)\ne 0$ by construction, and for any $x\in\Sigma$, we have $\phi(x)$ is among a finite list of nonnegative integers. Therefore, it follows from condition (i) of Definition \ref{deffamily} and the density of $\Sigma$ that $\phi^{-1}(\phi(x_0))$ and its complement are closed in $X$, and therefore $\phi$ is constant on some open and closed subset $X_0\subset X$ containing $x_0$. Thus for every $x\in  \Sigma\cap X_0$, we must have $\pi_{i,x}\cong Q(S_{i,x})$ for some multiset of segments $S_{i,x}$ with $S_{i,x}\geq\{\Delta_{1,i}\otimes\chi_{1,i,x},\dotsc,\Delta_{r,i}\otimes\chi_{r,i,x}\}$ for some unramified characters $\chi_{1,i,x},\dotsc,\chi_{r,i,x}$ of $F_i^\times$. This concludes Step 1.

\textit{Step 2.} We must now prove the reverse inequality. If it were the case that $N=1$ and that $\pi_{1,x_0}$ had supercuspidal support consisting of the trivial representation (equivalently, if $\pi_{1,x_0}$ had Iwahori fixed vectors), then we could prove this inequality by applying Theorem \ref{thmK1fixed} and Corollary \ref{corpfineq}; the dimension $\dim_{\C}(\pi_{1,x_0}^{K_{1,n_1}})$ as in Theorem \ref{thmK1fixed} is simply the trace of the Hecke operator $\vol(K_{1,n_1})^{-1}\chars(K_{1,n_1})$, and this trace must be locally constant in the family. In fact, we actually make this argument below after reducing to the case where our representations have trivial supercuspidal support. To make such a reduction, we invoke Arthur--Clozel base change in order to trivialize the supercuspidal supports as follows.

First, for each $i$ with $1\leq i\leq N$, let $F_i'$ be a finite Galois extension of $F_i$ with the following property: For all $x\in \Sigma\cap X_0$, if $(\rho_{i,x},N_{i,x})$ is the Weil--Deligne representation attached to $\pi_{i,x}$, then $\rho_{i,x}|_{I_{F_i'}}$ is trivial. Such an $F_i'$ exists because we just showed that, up to unramified twist, the supercuspidal support of $\pi_{i,x}$ is independent of $x\in \Sigma\cap X_0$.

Since finite Galois extensions of local fields are solvable, we may invoke the results of \cite[Chapter 1]{AC}, which give us representations $\pi_{i,x}'$ of $GL_{n_i}(F_i')$ which are the base change lifts of $\pi_{i,x}$ for any $x\in \Sigma\cap X_0$. By our choice of $F_i'$, the representations $\pi_{i,x}'$ may be described as follows. For each $k$ with $1\leq k\leq r$, write $S_{i,x}=\{\Delta_{1,i,x},\dotsc,\Delta_{r_x,i,x}\}$ for some $r_x$. For each $k=1,\dotsc,r_x$, let $\tau_{k,i,x}$ be the first supercuspidal representation in the segment $\Delta_{k,i,x}$. Say it is a representation of $GL_{n_{k,i,x}}(F_i)$ for some $n_{k,i,x}$. Then the base change $\tau_{k,i,x}'$ of $\tau_{k,i,x}$ to $F_i'$ is an unramified representation and therefore occurs as a constituent in the Borel induction of several unramified characters $\psi_{1,k,i,x},\dotsc,\psi_{n_{k,i,x},k,i,x}$ of $(F_i')^\times$. For any $m$ with $1\leq m\leq n_{k,i,x}$, let $\Delta_{m,k,i,x}'$ be the segment of length equal to that of $\Delta_{k,i,x}$ and starting with the character $\psi_{m,k,i,x}$. Let $S_{i,x}'$ be the multiset of segments consisting of the segments $\Delta_{m,k,i,x}'$ for $1\leq k\leq r$ and $1\leq m\leq n_{k,i,x}$. Then $\pi_{i,x}'\cong Q(S_{i,x}')$.

Now for any $i$ with $1\leq i\leq N$, let $\mc{H}_i'=C_c^\infty(GL_{n_i}(F_i'),\C)$, and let
\[\mc{H}'=C_c^\infty({\textstyle\prod_i}GL_{n_i}(F_i'),\C),\]
so that
\[\mc{H}'\cong\mc{H}_1'\otimes_\C\dotsb\otimes_\C\mc{H}_N'.\]
Then \cite[Chapter 1]{AC} also implies the following. Given any $f_i'\in\mc{H}_i'$, there is a transferred operator $f_i\in\mc{H}_i$ such that for any irreducible smooth representation $\sigma$ of $GL_{n_i}(F_i)$ with base change $\sigma'$ to $F_i'$, we have
\[\tr(f_i'|\sigma')=\tr(f_i|\sigma).\]
This transfer is linear, and so we can define a map $T':\mc{H}'\to R$ by
\[T'(f_1'\otimes\dotsb\otimes f_N')=T(f_1\otimes\dotsb\otimes f_N),\]
where each $f_i$ is associated with $f_i'$ as just described. One checks easily that $(X,\Sigma,R,T')$ is a family of irreducible smooth representations of $\prod_i GL_{n_i}(F_i')$ and that, furthermore, for any $x\in\Sigma$, we have
\[T_x'(f')=\tr(f'|\pi_{1,x}'\boxtimes\dotsb\boxtimes\pi_{N,x}').\]

Now fix $j$ with $1\leq j\leq N$. Let $x\in \Sigma\cap X_0$. By construction, for any $i$ with $1\leq i\leq N$, the representations $\pi_{i,x}'$ are constituents of unramified principal series, and therefore possess fixed vectors by the Iwahori subgroup $I_i'$ defined by
\[I_i'=\sset{g\in GL_{n_i}(\O_{F_i'})}{(g\textrm{ mod }\varpi_i')\in B(\O_{F_i'}/\varpi_i')},\]
where $\varpi_i'$ denotes a fixed uniformizer in $\O_{F_i'}$ and $B$ is the standard Borel in $GL_{n_i}$. Let $f_i'=\vol(I_i')^{-1}1_{I_i'}$ if $i\ne j$, and let $f_j'=\vol(K_{1,j,n_j}')^{-1}1_{K_{1,j,n_j}'}$, where
\[K_{1,j,n_j}'=\sset{g\in GL_{n_j}(\O_{F_j'})}{g\equiv 1\modulo{\varpi_j'}},\]
similarly to the group of the previous section. Set $f'=f_1'\otimes\dotsb\otimes f_N'$. Then we have
\[T_x'(f')=\dim_{\C}((\pi_{j,x}')^{K_{1,j,n_j}'})\prod_{i\ne j}\dim_{\C}((\pi_{i,x}')^{I_i'}).\]
Note that there are only finitely many possible values for the above expression, since
\[\dim_{\C}((\pi_{i,x}')^{I_i'})\leq n_i!\]
and
\[\dim_{\C}((\pi_{j,x}')^{K_{1,j,n_j}'})\leq \vert GL_{n_j}(\O_{F_j'}/\varpi_j')\vert.\]

Let $p_i$ denote the residue characteristic of $F_i$. We would like at this point to apply Theorem \ref{thmK1fixed} and Corollary \ref{corpfineq} to recover the lengths of the intervals comprising $S_{i,x}'$, but unfortunately the $p_j$-adic valuation of any of the terms $\dim_{\C}((\pi_{i,x}')^{I_i'})$ above could be nonzero. To overcome this, we make a second base change. Let $F_i''=F_i'$ if $i\ne j$, and let $F_j''$ be the unramified quadratic extension of $F_j'$. Then we get base changed representations $\pi_{i,x}''$ for any $i$ (which are just the representations $\pi_{i,x}'$ if $i\ne j$) and analogously defined Hecke algebras $\mc{H}_i''$ and $\mc{H}''$, and also an analogously defined map $T'':\mc{H}''\to R$. We consider the operators $f_i''$ defined by $f_i''=f_i'$ if $i\ne j$ and $f_j''=\vol(K_{1,j,n_j}'')^{-1}\chars(K_{1,j,n_j}'')$, where $K_{1,j,n_j}''$ is the analogously defined subgroup of $GL_{n_j}(\O_{F_j''})$. Finally, let $f''=f_1''\otimes\dotsb\otimes f_N''$. We then have
\[T_x''(f'')=\dim_{\C}((\pi_{j,x}'')^{K_{1,j,n_j}''})\prod_{i\ne j}\dim_{\C}((\pi_{i,x}')^{I_i'}),\]
for any $x\in \Sigma\cap X_0$

Now similarly to the conclusion of Step 1, after possibly shrinking $X_0$, we find that the value $T_x''(f'')/T_x'(f')$ is constant on $X_0$, and by Theorem \ref{thmK1fixed}, we have
\begin{equation}
\label{eqvpj}
\frac{v_{p_j}(T_x''(f'')/T_x'(f'))}{v_{p_j}(\vert\O_{F_j'}/\varpi_j'\vert)}=\sum_{k=1}^{r_x}n_{k,j,x}\frac{\ell_{k,j,x}(\ell_{k,j,x}-1)}{2},
\end{equation}
where $\ell_{k,j,x}$ is the length of the interval $\Delta_{k,j,x}$.

Now assume that $S_{j,x}>\{\Delta_{1,j}\otimes\chi_{1,j},\dotsc,\Delta_{r,j}\otimes\chi_{r,j}\}$ (strict inequality) for some unramified characters $\chi_{1,j},\dotsc,\chi_{r,j}$ of $F_j^\times$. It is not too difficult to see from the description we gave above that $S_{j,x}'$ is strictly greater than the multiset consisting of the segments $\Delta_{m,k,j,x_0}'\otimes\chi_{k,j}'$, $1\leq k\leq r$, $1\leq m\leq n_{k,j,x_0}$, where $\chi_{k,j}'$ is the base change of $\chi_{k,j}$ to $(F_j')^\times$; moreover, a similar statement holds for the base change to $F_j''$. By Corollary \ref{corpfineq}, we have
\[\sum_{k=1}^{r_x}n_{k,j,x}\frac{\ell_{i,x}(\ell_{i,x}-1)}{2}>\sum_{k=1}^{r_0}n_{k,j,x_0}\frac{\ell_{i,x_0}(\ell_{i,x_0}-1)}{2},\]
which contradicts the constancy of \eqref{eqvpj} on $\Sigma\cap X_0$. This contradiction completes Step 2, hence also the proof of the theorem.
\end{proof}

\section{Detecting twists with Hecke operators}
\label{secdetect}

In this section we make some technical preparation for the results on epsilon factors in Section \ref{secepsilon}. We first set some notation. We fix throughout this section a positive integer $n$, a nonarchimedean local field $F$ of characteristic zero and a uniformizer $\varpi\in F$. Also, we set the following.

\begin{notation}
Let $z\in\C^\times$. We denote by $\chi^{(z)}$ the unramified character of $F^\times$ such that $\chi^{(z)}(\varpi)=z$.
\end{notation}

We also need the following notion.

\begin{definition}
\label{deftwind}
Let $\pi$ be an irreducible smooth representation of $GL_n(F)$. We write $I(\pi)$ for the largest integer $M$ such that there exists a primitive $M$th root of unity $\zeta$ for which $\pi\cong\pi\otimes\chi^{(\zeta)}$. We call $I(\pi)$ the \textit{twisting index} of $\pi$.
\end{definition}

The twisting index $I(\pi)$ of $\pi$ always exists and divides the integer $n$. This follows easily from the fact that if $\omega$ is the central character of $\pi$, then for any unramified character $\chi$ of $F^\times$, the central character of $\pi\otimes\chi$ is $\omega\chi^n$.

We now recall some notions from the theory of types. Let $\tau$ be a supercuspidal representation of $GL_n(F)$. According to \cite{BKbook}, there exists an open compact subgroup $K_\rho\subset GL_n(\O_F)$ and an irreducible smooth representation $\rho$ of $K_\rho$ such that, for any irreducible smooth representation $\pi$ of $GL_n(F)$, the space $\hom_{K_\rho}(\rho,\pi)$ is nontrivial if and only if $\pi$ is an unramified twist of $\tau$. We call $(K_\rho,\rho)$ a \textit{type} for $\tau$. Then $\rho$ gives rise to an idempotent $e_\rho$ in the Hecke algebra $\mc{H}=C_c^\infty(GL_n(F),\C)$, a formula for which is given by
\[e_\rho(g)=\begin{cases}
\frac{\dim\rho}{\vol(K_\rho)}\tr(g^{-1}|\rho)&\textrm{if }g\in K_\rho;\\
0&\textrm{otherwise.}
\end{cases}\]
We then define the algebra $\mc{H}(\rho)$ by
\[\mc{H}(\rho)=e_\rho\mc{H}e_\rho.\]
According to \cite[(5.5)(b), Proposition 5.6]{BKcovers}, the type $(K_\rho,\rho)$ may be constructed so that $\mc{H}(\rho)$ is canonically identified with the ring of regular functions on the Bernstein variety of $\tau$; in particular, it is a commutative integral domain whose characters $\mc{H}(\rho)\to\C$ are in bijection with the isomorphism classes of unramified twists of $\tau$. Moreover, if $\pi$ is an unramified twist of $\tau$, then the $\rho$-isotypic component $\pi[\rho]$ of $\pi$ is isomorphic to $\rho$ and $\mc{H}(\rho)$ acts as scalars on $\pi[\rho]$ through the character of $\mc{H}(\rho)$ determined by $\pi$.

To proceed, we need the following definition.

\begin{definition}
Let $f\in\mc{H}=C_c^\infty(GL_n(F),\C)$. We say $f$ is \textit{homogeneous of degree} $d$ if for all $g$ in the support of $f$, we have $\det(g)\in\varpi^d\O_F^\times$.
\end{definition}

It is easy to see that if $f$ and $f'$ are elements of $C_c^\infty(GL_n(F),\C)$ which are homogeneous of degrees $d$ and $d'$, respectively, then $ff'$ (product computed as usual via convolution) is homogeneous of degree $d+d'$.

\begin{lemma}
\label{lemheckehomog}
As above, let $\tau$ be a supercuspidal representation of $GL_n(F)$ and $(K_\rho,\rho)$ a type for $\tau$. For any integer $d$, let $\mc{H}(\rho)_d$ be the set of homogeneous elements in $\mc{H}(\rho)$ of degree $d$. Then $\mc{H}(\rho)$ is $\Z$-graded with graded pieces $\mc{H}(\rho)_d$.
\end{lemma}

\begin{proof}
Let $K_\rho'\subset K_\rho$ be the kernel of $\rho$. Then the Hecke algebra
\[C_c^\infty(K_\rho'\backslash GL_n(F)/K_\rho',\C)\]
is easily seen to be graded by degree, since it is free as a $\C$-vector space on the homogeneous elements $\chars(K_\rho' gK_\rho')$, where the elements $g$ run over representatives for the double coset space $K_\rho'\backslash GL_n(F)/K_\rho'$. Since $e_\rho$ is homogeneous of degree $0$ (being supported on $K_\rho$), the elements $e_\rho\chars(K_\rho' gK_\rho')e_\rho$ are also homogeneous. But these elements generate $\mc{H}(\rho)$, and this is enough to conclude.
\end{proof}

\begin{lemma}
\label{lemtwindideal}
Still as above, let $\tau$ be a supercuspidal representation of $GL_n(F)$ and $(K_\rho,\rho)$ a type for $\tau$. Let $c:\mc{H}(\rho)\to\C$ be any character. Let $J$ be the set of integers $d$ for which there exists an element $f\in\mc{H}(\rho)_d$ such that $c(f)\ne 0$. Then $J$ is the ideal in $\Z$ generated by the twisting index $I(\tau)$ of $\tau$.
\end{lemma}

\begin{proof}
Let $\pi$ be the unramified twist of $\tau$ corresponding to $c$. Clearly $I(\pi)=I(\tau)$. Let $m\in\Z$ be an integer, and let $\chi$ be an unramified character of $F^\times$. Then if we view the space of $\pi\otimes\chi$ as the same as the space of $\pi$ but with the twisted action, then $(\pi\otimes\chi)[\rho]$ is identified with $\pi[\rho]$; this is just because $\chi$ is trivial on $K_\rho$. It follows that if $v$ is in the space of $(\pi\otimes\chi)[\rho]$ and if $f\in\mc{H}(\rho)_m$, then
\begin{multline}
\label{eqnhomogchi}
(\pi\otimes\chi)(f)v=\int_{\Supp(f)}f(g)\left((\pi\otimes\chi)(g)\right)v\,dg\\
=\chi(\varpi^m)\int_{\Supp(f)}f(g)\pi(g)v\,dg=\chi(\varpi)^m c(f)v.	
\end{multline}
In the case that $\chi=\chi^{(\zeta)}$ for $\zeta$ a primitive $I(\tau)$th root of unity, we have that \eqref{eqnhomogchi} gives
\[c(f)v=\zeta^m c(f)v,\]
because in this case $\pi\otimes\chi\cong\pi$. Thus if $I(\tau)\nmid m$, then $c(f)=0$, i.e., we have $J\subset I(\tau)\Z$. We now show the opposite inclusion.

To do this, let $k>1$ be any integer. We claim that $J$ is not contained in $I(\tau)k\Z$. Indeed, assume otherwise. Then \eqref{eqnhomogchi}, in the case that $\chi$ is exact order $I(\tau)k$, gives $(\pi\otimes\chi)(f)v=\pi(f)v$; i.e., the characters of $\mc{H}(\rho)$ corresponding to $\pi$ and $\pi\otimes\chi$ are the same. It follows that $\pi\cong\pi\otimes\chi$, a contradiction to the fact that $I(\pi)=I(\tau)$.

Now $J$ is closed under addition because $c$ is an algebra homomorphism, and so the claim just proved implies by elementary considerations that $J$ contains integers $d_+,d_-$ with $d_\pm\equiv \pm I(\tau)\modulo{n}$ (same sign). Let $a_{\pm}=(\pm I(\tau)-d_{\pm})/n$, and let
\[g_{\pm}=\diag(\varpi^{a_{\pm}},\dotsc,\varpi^{a_{\pm}})\in GL_n(F).\]
Let $f_{\pm}\in\mc{H}(\rho)$ be a homogeneous element of degree $d_\pm$ such that $c(f_{\pm})\ne 0$. Then because $g_\pm$ is central, the Hecke operator $f_{\pm}'$ defined by $f_{\pm}'(g)=f_{\pm}(gg')$ is in $\mc{H}(\rho)$, and in fact it is homogeneous of degree $\pm I(\tau)$. Moreover, for any nonzero $v$ in the space of $\pi[\rho]$, we have
\[\pi(f_{\pm}')v=\pi(g_{\pm})\pi(f_{\pm})v=\chi(\varpi^{a_{\pm}})^n c(f_\pm)v\ne 0.\]
Thus $\pm I(\tau)\in J$, which implies $J\supset I(\tau)\Z$, as desired.
\end{proof}

We now study Jacquet modules of certain induced representations. (Our conventions will be such that our Jacquet modules are not normalized by a modulus character, since this is what is done in \cite{BKcovers}.) We set some notation.

\begin{notation}
\label{notnroots}
Let $P\subset GL_n$ be a standard parabolic subgroups with standard Levi decompositions $P=MN$.
\begin{itemize}
\item We denote by $T_n$ the standard diagonal torus in $GL_n$.
\item We denote by $\Psi_n$ the set of roots for $T_n$ in $GL_n$, and we denote by $\Psi_n^+$ and $\Psi_n^-$ the standard sets of positive and negative roots in $\Psi_n$, according to the sign.
\item We write $\Psi_M$ for the subset of $\Psi_n$ of roots in $M$, and similarly for $\Psi_M^+$ and $\Psi_M^-$.
\item We write $W_n$ for the Weyl group of $T$ in $GL_n$, and $W_M$ for the subgroup of $W_n$ generated by reflections about the roots in $\Psi_M$.
\item If $P'\subset GL_n$ is another standard parabolic subgroup with standard Levi $M'$, then we write $[W_{M'}\backslash W_n/W_M]$ for the set of minimal length (with respect to $\Psi_n^+$) representatives in $W_n$ of the double coset space $W_{M'}\backslash W_n/W_M$. Equivalently this is the set of $w\in W_n$ for which $w\Psi_M^+\subset\Psi_n^+$ and $w^{-1}\Psi_{M'}^+\subset\Psi_n^+$.
\item We denote the modulus character of $P$ by $\delta_P$.
\item If $\pi$ is a smooth admissible representation of $GL_n(F)$, then the Jacquet module $\Jac_{M(F)}(\pi)$ is the smooth admissible representation of $M(F)$ given by the space of $N(F)$-coinvariants of $\pi$ with the natural action (\textit{without} an extra twist by $\delta_P^{-1/2}$).
\end{itemize}
\end{notation}

We have the following tool for computing the semisimplifications of Jacquet modules.

\begin{proposition}
\label{propcasselman}
Let $P=MN$ and $P'=M'N'$ be two standard parabolics in $GL_n$. For any $w\in [W_{M'}\backslash W_n/W_M]$ (see Notation \ref{notnroots}), define a character $\delta_w$ of $w^{-1}M'(F)w\cap M(F)$ to be the modulus of the unique character of $w^{-1}M'w\cap M$ whose restriction to $T_n$ is given by the following product of roots:
\[\prod_{\substack{\alpha\in\Psi_n^+\backslash\Psi_M^+ \\ w\alpha\in\Psi_n^-\backslash\Psi_{M'}^-}}\alpha\prod_{\substack{\alpha\in\Psi_n^-\cup\Psi_M \\ w\alpha\in\Psi_n^-\backslash\Psi_{M'}^-}}\alpha^{-1}.\]
Let $\pi$ be a smooth admissible representation of $M'(F)$. Then
\begin{multline*}
\Jac_{M(F)}(\Ind_{P'(F)}^{GL_n(F)}(\pi))^{\sss}\cong\\
\bigoplus_{w\in [W_{M'}\backslash W_n/W_M]}\mathrm{unInd}_{w^{-1}P'(F)w\cap M(F)}^{M(F)}(w^{-1}(\Jac_{M'(F)\cap wM(F)w^{-1}}(\pi))\delta_w^{1/2})^{\sss},
\end{multline*}
where $(\cdot)^{\sss}$ denotes semisimplification, where $\mathrm{unInd}$ denotes unnormalized parabolic induction, and where for a representation $\sigma$ of $M'(F)\cap wM(F)w^{-1}$, we define $w^{-1}(\sigma)$ to be the representation of $w^{-1}M'(F)w\cap M(F)$ on the space of $\sigma$ given by $(w^{-1}(\sigma))(g)=\sigma(wgw^{-1})$; here we implicitly choose a lift of $w$ to the normalizer $N_{GL_n(F)}(T(F))$ of $T(F)$ in $GL_n(F)$ in order to compute the product $w^{-1}gw$ for $g\in GL_n(F)$.
\end{proposition}

\begin{proof}
This follows immediately from \cite[Propositions 6.3.1, 6.3.3]{casselman}.
\end{proof}

Before we apply this proposition, we will first need the following combinatorial lemma.

\begin{lemma}
\label{lemlevicomb}
Let $\ell_1,\dotsc,\ell_r$ and $m$ be positive integers where $r>0$, and write $n=\sum_{j=1}^r \ell_j$. Let $M$ be the standard Levi in $GL_{nm}$ with $M=GL_{m}\times\dotsb\times GL_{m}$ ($n$ times), and let $M'$ be the standard Levi in $GL_{nm}$ with $M'=GL_{m\ell_1}\times\dotsb\times GL_{m\ell_r}$. Then if $w\in[W_{M'}\backslash W_{nm}/W_M]$ (see Notation \ref{notnroots}) is such that $wMw^{-1}\ne M$, then $w^{-1}M'w\cap M\ne M$.
\end{lemma}

\begin{proof}
Identify $W_{nm}$ with the symmetric group $S_{nm}$ on $\{1,\dotsc,nm\}$ in the standard way. We will write as is usual $e_i-e_j$, with $1\leq i,j\leq nm$ and $i\ne j$, for an arbitrary root of $GL_{nm}$. Then under the identification of $W_{nm}$ with $S_{nm}$, we have $w(e_i-e_j)=e_{w(i)}-e_{w(j)}$ for any $w\in W_{nm}$.

Now let $w\in[W_{M'}\backslash W_{nm}/W_M]$ with $wMw^{-1}\ne M$. We note that the condition that $w\in[W_{M'}\backslash W_{nm}/W_M]$ means that $w\alpha>0$ for any positive root $\alpha$ in $M$, and $w^{-1}\beta>0$ for any positive root $\beta$ in $M'$. Then since $wMw^{-1}\ne M$, we have $m>1$, and there is some simple root $\alpha$ in $M$ for which $w\alpha$ is not in $M$. Write $\alpha=e_{am+b}-e_{am+b+1}$ for some $a\in\{0,\dotsc,n-1\}$ and $b\in\{1,\dotsc,m-1\}$. Write $n_i=\sum_{j=1}^i\ell_j$ for any $i=0,\dotsc,r$, so that $n_0=0$ and $n_r=n$. Let $i\in\{0,\dotsc,r-1\}$ be such that $n_i m+1\leq w(am+b)\leq n_{i+1}m$. Then if $w(am+b+1)$ does not also satisfy $n_i m+1\leq w(am+b+1)\leq n_{i+1}m$, then $w(am+b+1)>n_{i+1}m$ since $w(am+b+1)>w(am+b)$. But then $w\alpha\notin M'$, from which it follows that $w^{-1}M'w\cap M\ne M$.

Thus we will now assume the following for sake of contradiction: For any $a\in\{0,\dotsc,n-1\}$ and $b\in\{1,\dotsc,m-1\}$ such that $w(e_{am+b}-e_{am+b+1})$ is not in $M$, if $i\in\{0,\dotsc,r-1\}$ is such that $n_i m+1\leq w(am+b)\leq n_{i+1}m$, then also $n_i m+1\leq w(am+b+1)\leq n_{i+1}m$. Given such $a$, $b$ and $i$, we claim that $w(am+b+1)=w(am+b)+1$. Indeed, all the roots
\[e_{w(am+b)}-e_{w(am+b)+1},e_{w(am+b)+1}-e_{w(am+b)+2},\dotsc,e_{w(am+b+1)-1}-e_{w(am+b+1)}\]
are positive, and they are in $M'$ by our assumption. Write $c=w(am+b+1)-w(am+b)$. We then have that
\[\sum_{j=1}^{c}w^{-1}(e_{w(am+b)+j}-e_{w(am+b)+j-1})=e_{am+b}-e_{am+b-1}\]
is a sum of positive roots. Since $e_{am+b}-e_{am+b-1}$ is simple, the sum has only one term, i.e., $w(am+b+1)=w(am+b)+1$, as claimed.

Now since $w(e_{am+b}-e_{am+b+1})=e_{w(am+b)}-e_{w(am+b)+1}$ is not in $M$, there is an $a'\in\{n_i+1,\dotsc,n_{i+1}\}$ such that $w(am+b)=a'm$, and in fact $a'\ne n_{i+1}$ since $w(am+b+1)=a'm+1\leq n_{i+1}m$.

Let us now assume our $a$ and $b$ with $w(e_{am+b}-e_{am+b+1})$ not in $M$ had been chosen so that our $i$ has the property that $n_{i+1}m-w(am+b)=(n_{i+1}-a')m$, or equivalently, that $n_{i+1}m-w(am+b+1)=(n_{i+1}-a')m-1$, is minimal. Since
\[w(am+b+1)<w(am+b+2)<\dotsb<w((a+1)m),\]
each of
\[e_{w(am+b+1)}-e_{w(am+b+2)},e_{w(am+b+2)}-e_{w(am+b+3)},\dotsc,e_{w((a+1)m-1)}-e_{w((a+1)m)}\]
must be in $M$ by our assumption on minimality. They are moreover simple by the summation argument above, so this forces 
\begin{equation}
\label{eqnwamb}
w(am+b+1)=a'm+1,\, w(am+b+2)=a'm+2,\dotsc,\, w((a+1)m)=a'm+(m-b).
\end{equation}
Now $w^{-1}(a'm+(m-b+1))>(a+1)m$, so $w^{-1}(a'm+(m-b+1))$ must be of the form $a''m+1$ with $a''>a$; otherwise $e_{w^{-1}(a'm+(m-b+1))-1}-e_{w^{-1}(a'm+(m-b+1))}$ is in $M$, but $w(e_{w^{-1}(a'm+(m-b+1))-1}-e_{w^{-1}(a'm+(m-b+1))})=e_{w(w^{-1}(a'm+(m-b+1))-1)}-e_{a'm+(m-b+1)}$ cannot be in $M$ by \eqref{eqnwamb}, contradicting our minimality assumption. The aforementioned summation argument once again forces
\begin{multline*}
w(a''m+1)=a'm+(m-b+1),\, w(a''m+2)=a'm+(m-b+2),\dotsc,\\ w((a''+1)m)=(a'+1)m+(m-b).
\end{multline*}
But then in particular we have $w(e_{a''m+b}-e_{a''m+b+1})=e_{(a'+1)m}-e_{(a'+1)m+1}$, which is not in $M$, a final contradiction to our minimality assumption, ultimately contradicting the original assumption we made on $w$. This completes the proof.
\end{proof}

We now use Proposition \ref{propcasselman} to study representations whose supercuspidal support consists of copies of the same representation.

Fix a supercuspidal representation $\tau$ of $GL_{n_\tau}(F)$ for some $n_\tau\geq 1$. Let $\Delta_1,\dotsc,\Delta_r$ be segments of the form
\[\Delta_i=\{\tau,\tau\otimes\vert\det\vert,\dotsc,\tau\otimes\vert\det\vert^{\ell_i-1}\},\]
with $\ell_i$ the length of $\Delta_i$. Let $\chi_1,\dotsc,\chi_r$ be unramified characters of $F^\times$. For each $i=1,\dotsc,r$, let $n_i=\sum_{j=1}^i\ell_j$, and write $n_0=0$ and $n=n_r$. Let $P\subset GL_{nn_\tau}$ be the standard parabolic with Levi $M=GL_{n_\tau}\times\dotsb\times GL_{n_\tau}$ ($n$ times), and $P'\subset GL_{nn_\tau}$ the standard parabolic with Levi $M'=GL_{\ell_1 n_\tau}\times\dotsb\times GL_{\ell_r n_\tau}$. We will consider the induced representation
\begin{equation}
\label{eqnindP'}
\Ind_{P'(F)}^{GL_{nn_\tau}(F)}((Q(\Delta_1)\otimes\chi_1)\boxtimes\dotsc\boxtimes(Q(\Delta_r)\otimes\chi_r))
\end{equation}
and its Jacquet module down to $M(F)$.

\begin{lemma}
\label{lemsimplejac}
Let the notation be as above. Identify the Weyl group of $GL_{nn_\tau}$ with the symmetric group $S_{nn_\tau}$ in the standard way. Let $S'\subset S_{nn_\tau}$ be the subset of permutations $w\in S_{nn_\tau}$ with the following properties: First, for any $a\in\{0,\dotsc,n-1\}$ and $b\in\{1,\dotsc,n_\tau\}$, we require $w(an_\tau+b)=w(an_\tau)+b$; second, for any $i\in\{0,\dotsc,r-1\}$, we require $w^{-1}((n_{i}+1)n_\tau)<w^{-1}((n_{i}+2)n_\tau)<\dotsb<w^{-1}(n_{i+1} n_\tau)$. Then the semisimplification of the (unnormalized) Jacquet module is given by
\[\Jac_{M(F)}(\Ind_{P'(F)}^{GL_{nn_\tau}(F)}((Q(\Delta_1)\otimes\chi_1)\boxtimes\dotsc\boxtimes(Q(\Delta_r)\otimes\chi_r)))^{\sss}\cong\bigoplus_{w\in S'}w^{-1}(\sigma)\otimes\delta_P^{1/2},\]
where $\sigma$ is the representation of $M$ given by
\begin{multline*}
\sigma=\,(\tau\otimes\chi_1\otimes\vert\det\vert^{\ell_1-1})\boxtimes(\tau\otimes\chi_1\otimes\vert\det\vert^{\ell_1-2})\boxtimes\dotsb\boxtimes(\tau\otimes\chi_1)\\
\begin{aligned}
&\boxtimes(\tau\otimes\chi_2\otimes\vert\det\vert^{\ell_2-1})\boxtimes(\tau\otimes\chi_2\otimes\vert\det\vert^{\ell_2-2})\boxtimes\dotsb\boxtimes(\tau\otimes\chi_2)\\
&\dotsb\\
&\boxtimes(\tau\otimes\chi_r\otimes\vert\det\vert^{\ell_r-1})\boxtimes(\tau\otimes\chi_r\otimes\vert\det\vert^{\ell_r-2})\boxtimes\dotsb\boxtimes(\tau\otimes\chi_r).
\end{aligned}
\end{multline*}
(See Proposition \ref{propcasselman} for the meaning of $w^{-1}(\sigma)$.)
\end{lemma}

\begin{proof}
Write $\pi$ for the representation of \eqref{eqnindP'}. We first note that $\pi$ is a constituent of a representation of the form
\[\Ind_{P(F)}^{GL_{nn_\tau}(F)}((\tau\boxtimes\dotsb\boxtimes\tau)\otimes\chi)\]
for some unramified character $\chi$ of the center of $M(F)$.

Let $w\in [W_{M'}\backslash W_{nn_\tau}/W_M]$. Assume that $wMw^{-1}\ne M$. Write $P''$ for the standard parabolic in $GL_{nn_\tau}$ with Levi $M'\cap wMw^{-1}$. Then by Proposition \ref{propcasselman} applied with $P''$ and $P$ in place of $P$ and $P'$, respectively, we see that the Jacquet module
\[\Jac_{M'(F)\cap wM(F)w^{-1}}(\pi)\]
is (up to semisimplification) parabolically induced from unramified twists of representations of the form
\begin{equation}
\label{eqw'jac}
(w')^{-1}\Jac_{M(F)\cap w'(M'(F)\cap wM(F)w^{-1})(w')^{-1}}((\tau\boxtimes\dotsb\boxtimes\tau)\otimes\chi_w),
\end{equation}
where $w'\in [W_{M}\backslash W_{nn_\tau}/W_{M'\cap wMw^{-1}}]$ and $\chi_w$ is some unramified character (depending on $\chi$ and $w$). Since $wMw^{-1}\ne M$, Lemma \ref{lemlevicomb} implies that $M'\cap wMw^{-1}$ is strictly contained in $wMw^{-1}$, and in particular has dimension smaller than that of $M$. Therefore $M\cap w'(M'\cap wMw^{-1})(w')^{-1}$ is strictly contained in $M$. Thus the Jacquet module in \eqref{eqw'jac} vanishes because $(\tau\boxtimes\dotsb\boxtimes\tau)\otimes\chi_w$ is supercuspidal. So $\Jac_{M'(F)\cap wM(F)w^{-1}}(\pi)=0$ if $wMw^{-1}\ne M$.

Combining this fact with Proposition \ref{propcasselman} shows that
\begin{multline}
\label{eqjacM1}
\Jac_{M(F)}(\pi)^{\sss}\\
=\bigoplus_{\substack{w\in [W_{M'}\backslash W/W_M]\\ wMw^{-1}=M}} w^{-1}(\Jac_{M(F)}((Q(\Delta_1)\otimes\chi_1)\boxtimes\dotsc\boxtimes(Q(\Delta_r)\otimes\chi_r)))\otimes\delta_w^{1/2},
\end{multline}
The set $[W_{M'}\backslash W_{nn_\tau}/W_M]$ is equal to the set of all $w\in W$ for which $w\alpha>0$ for any positive root $\alpha$ in $M$ and for which $w^{-1}\beta>0$ for any positive root $\beta$ in $M'$, i.e., under the identification $W=S_{nn_\tau}$, it is the set of permutations $w$ for which
\begin{equation}
\label{eqwantauorder}
w(an_\tau+1)<w(an_\tau+2)<\dotsb<w(an_\tau+n_\tau)
\end{equation}
for any $a\in\{0,\dotsc,n-1\}$ and for which
\[w^{-1}((n_{i}+1)n_\tau)<w^{-1}((n_{i}+2)n_\tau)<\dotsb<w^{-1}(n_{i+1} n_\tau)\]
for any $i\in\{0,\dotsc,r-1\}$. In the presence of the condition \eqref{eqwantauorder}, the condition $wMw^{-1}=M$ on such a $w$ exactly amounts to the condition that $w(an_\tau+b)=w(an_\tau)+b$ for any $a\in\{0,\dotsc,n-1\}$ and $b\in\{1,\dotsc,n_\tau-1\}$. Thus the index set of the sum in \eqref{eqjacM1} is exactly $S'$.

Now for $i=1,\dotsc,r$, let $P_i$ be the standard parabolic subgroup of $GL_{n_in_\tau}$ with Levi $M_i=GL_{n_\tau}\times\dotsb\times GL_{n_\tau}$ ($n_i$ times). Then the Jacquet module of the generalized Steinberg representation $Q(\Delta_i)$ to $M_i(F)$ along $P_i(F)$ is given by $(\tau\boxtimes\dotsb\boxtimes\tau)\otimes(\vert\det\vert^{(\ell_i-1)/2}\delta_{P_i})$. Thus
\[\Jac_{M(F)}((Q(\Delta_1)\otimes\chi_1)\boxtimes\dotsc\boxtimes(Q(\Delta_r)\otimes\chi_r))\cong\sigma\otimes(\delta_{P_1}^{1/2}\boxtimes\dotsb\boxtimes\delta_{P_r}^{1/2}).\]
In view of \eqref{eqjacM1}, we are therefore done if we show that
\begin{equation}
\label{eqmodchars}
\delta_{P}=w^{-1}(\delta_{P_1}\boxtimes\dotsb\boxtimes\delta_{P_r})\delta_w.
\end{equation}
So we must compute $\delta_w$.

Now $\delta_w$ is by definition the modulus of the character given by
\[\prod_{\substack{\alpha\in\Psi_{nn_\tau}^+\backslash\Psi_M^+ \\ w\alpha\in\Psi_{nn_\tau}^-\backslash\Psi_{M'}^-}}\alpha\prod_{\substack{\alpha\in\Psi_{nn_\tau}^-\cup\Psi_M \\ w\alpha\in\Psi_{nn_\tau}^-\backslash\Psi_{M'}^-}}\alpha^{-1}.\]
Since $wMw^{-1}=M$ and $M\subset M'$, a root $\alpha$ with $w\alpha\notin\Psi_{M'}$ can never have $\alpha\in\Psi_M$. Thus we may rewrite this as
\begin{align*}
\prod_{\substack{\alpha\in\Psi_{nn_\tau}^+ \\ w\alpha\in\Psi_{nn_\tau}^-\backslash\Psi_{M'}^-}}\alpha\prod_{\substack{\alpha\in\Psi_{nn_\tau}^- \\ w\alpha\in\Psi_{nn_\tau}^-\backslash\Psi_{M'}^-}}\alpha^{-1}&=\prod_{\substack{\alpha\in\Psi_{nn_\tau}^+ \\ w\alpha\in\Psi_{nn_\tau}^-\backslash\Psi_{M'}^-}}\alpha\prod_{\substack{\alpha\in\Psi_{nn_\tau}^+ \\ w\alpha\in\Psi_{nn_\tau}^+\backslash\Psi_{M'}^+}}\alpha\\
&=\prod_{\substack{\alpha\in\Psi_{nn_\tau}^+ \\ w\alpha\in\Psi_{nn_\tau}\backslash\Psi_{M'}}}\alpha=\prod_{\substack{\alpha\in\Psi_{nn_\tau}^+ \\ \alpha\notin w^{-1}(\Psi_{M'})}}\alpha.
\end{align*}
Since $w^{-1}$ preserves the sign of any root in $\Psi_{M'}$, this can be rewritten as
\[\prod_{\substack{\alpha\in\Psi_{nn_\tau}^+ \\ \alpha\notin w^{-1}(\Psi_{M'}^+)}}\alpha=\prod_{\alpha\in\Psi_{nn_\tau}^+}\alpha\prod_{\alpha\in\Psi_{M'}^+}(w^{-1}\alpha)^{-1}.\]
The modulus of this character is then clearly (the restriction to $T(F)$ of) $w^{-1}(\delta_{P_1}\boxtimes\dotsb\boxtimes\delta_{P_r})\delta_{P}$, which proves \eqref{eqmodchars} and also the lemma.
\end{proof}

\begin{lemma}
\label{lemjacisotypic}
Retain the notation of the previous lemma. Let $K\subset GL_{n_\tau}(F)$ be a compact open subgroup, and for $j=1,\dotsc,n$, let $K^{(j)}\subset M(F)=GL_{n_\tau}(F)^n$ be the subgroup
\[K^{(j)}=K\times\dotsb\times K\times \{1\}\times K\times\dotsb\times K,\]
where the trivial group $\{1\}$ is inserted in the $j$th slot. Let $\rho$ be an irreducible smooth representation of $K$, and let $\rho^{(j)}$ be the representation of $K^{(j)}$ given by
\[\rho^{(j)}=\rho\boxtimes\dotsb\boxtimes\rho\boxtimes\triv\boxtimes\rho\boxtimes\dotsb\boxtimes\rho,\]
where the trivial representation $\triv$ of the trivial group is in the $j$th slot. Then the $\rho^{(j)}$-isotypic component in the semisimplified Jacquet module,
\[\Jac_{M(F)}(\Ind_{P(F)}^{GL_{nn_\tau}(F)}((Q(\Delta_1)\otimes\chi_1)\boxtimes\dotsc\boxtimes(Q(\Delta_r)\otimes\chi_r)))^{\sss}[\rho^{(j)}],\]
which is naturally a representation of $GL_{n_\tau}(F)$ via the action of the $j$th copy of $GL_{n_\tau}(F)$ in $M(F)$, is isomorphic to
\[\bigoplus_{i=1}^r\bigoplus_{k=1}^n(\tau\otimes\chi_i)^{\oplus d_\rho^{n-1} m(\ell_i,k)}\otimes\vert\det\vert^{\ell_i-k+\frac{n+1}{2}-j},\]
where
\[d_\rho=\dim_{\C}(\tau[\rho]),\]
and
\[m(\ell_i,k)=(n-\ell_i)!\ell_i!{j-1 \choose k-1} {n-j \choose \ell_i-k}\prod_{i'=1}^{r}\frac{1}{\ell_{i'}!},\]
which is an integer.
\end{lemma}

\begin{proof}
We first partition the set $S'$ of Lemma \ref{lemsimplejac} into several pieces. For $i=1,\dotsc,r$ and $k$ a positive integer, let
\[S_j'(i,k)=\sset{w\in S'}{w(jn_\tau)=(n_{i-1}+k)n_\tau}\]
if $1\leq k\leq \ell_i$, and $S_j'(i,k)=\emptyset$ if $k>\ell_i$. Then $S'=\bigcup_{i=1}^r\bigcup_{k=1}^n S_j'(i,k)$ (disjoint union), and so by Lemma \ref{lemsimplejac}, if we write $\pi$ for the representation of \eqref{eqnindP'}, then
\[\Jac_{M(F)}(\pi)^{\sss}\cong\bigoplus_{i=1}^r\bigoplus_{k=1}^n\bigoplus_{w\in S_j'(i,k)}w^{-1}(\sigma)\otimes\delta_P^{1/2},\]
where $\sigma$ is as in Lemma \ref{lemsimplejac}. Thus
\[\Jac_{M(F)}(\pi)^{\sss}[\rho^{(j)}]\cong\bigoplus_{i=1}^r\bigoplus_{k=1}^n\bigoplus_{w\in S_j'(i,k)}(w^{-1}(\sigma)\otimes\delta_P^{1/2})[\rho^{(j)}].\]

Now fix $i$ and $k$. Let $\sigma_{i,k}$ be the representation of $GL_{n_\tau}$ which is given by $(n_{i-1}+k)$th component of the tensor product in the definition of $\sigma$. Note that every component of the tensor product defining $\sigma$ is an unramified twist of $\tau$, and therefore their $\rho$-isotypic components are all isomorphic to $\tau[\rho]$. Thus
\[(w^{-1}(\sigma)\otimes\delta_P^{1/2})[\rho^{(j)}]\cong (\sigma_{i,k}\otimes\vert\det\vert^{\frac{n+1}{2}-j})\boxtimes\tau[\rho]^{\boxtimes(n-1)},\]
because $\delta_P$ acts as $\vert\det\vert^{\frac{n+1}{2}-j}$ on the $j$th component of $M$. We note that the right hand side of this expression is independent of $w\in S_j'(i,k)$. Thus, using the definition of $\sigma$ we get
\[\Jac_{M(F)}(\pi)^{\sss}[\rho^{(j)}]\cong\bigoplus_{i=1}^r\bigoplus_{k=1}^n(\tau\otimes\chi_i\otimes\vert\det\vert^{\ell_i-k+\frac{n+1}{2}-j})^{\oplus\#S_j'(i,k)}\boxtimes\tau[\rho]^{\boxtimes(n-1)}.\]
We are therefore done if we can show that $\#S_j'(i,k)=m(\ell_i,k)$.

Now this equality is obvious if $k>\ell_i$, since both sides are zero. So assume $1\leq k\leq \ell_i$. Let us rewrite
\[S_j'(i,k)=\sset{w\in S'}{jn_\tau=w^{-1}((n_{i-1}+k)n_\tau)}.\]
Note that any $w\in S'$ is determined by its action on multiples of $n_\tau$. Since for $w\in S'$, we have that $w^{-1}$ is increasing on $\{(n_{i-1}+1)n_\tau,(n_{i-1}+2)n_\tau,\dotsc,n_in_\tau\}$, if $w^{-1}((n_{i-1}+k)n_\tau)=jn_\tau$, then there are ${j-1 \choose k-1}$ ways to choose the numbers
\[w^{-1}((n_{i-1}+1)n_\tau),\dotsc,w^{-1}((n_{i-1}+k-1)n_\tau).\]
Likewise there are ${n-j \choose \ell-k}$ ways to choose the numbers
\[w^{-1}((n_{i-1}+k+1)n_\tau),\dotsc,w^{-1}(n_in_\tau)\]
for such a $w$. The rest of the $(n-\ell_i)$ numbers in $\{n_\tau,2n_\tau,\dotsc,nn_\tau\}$ may be sent to any of the other numbers in this set besides the numbers $w^{-1}((n_{i-1}+1)n_\tau),\dotsc,w^{-1}(n_in_\tau)$, as long as $w^{-1}$ is increasing on $\{(n_{i'-1}+1)n_\tau,(n_{i'-1}+2)n_\tau,\dotsc,n_{i'}n_\tau\}$ for any $i'\ne i$. There are $(n-\ell_i)!$ ways to assign such values for $w^{-1}$, and groups of size $\prod_{i'\ne i}\ell_{i'}!$ are identified once these restrictions on order are imposed.

Thus there are
\[{j-1 \choose k-1} {n-j \choose \ell_i-k}(n-\ell_i)!\prod_{\substack{1\leq i'\leq r\\ i'\ne i}}\frac{1}{\ell_{i'}!}\]
elements in $S_j'(i,k)$, whence $\#S_j'(i,k)=m(\ell_i,k)$, as desired
\end{proof}

\begin{lemma}
\label{lemdeterminant}
Let $n\geq 1$ be an integer and $u$ a nonzero real number. Let $d$ be an integer with $1\leq d\leq n$. Let $A$ be the $d\times d$ matrix with entries
\[A_{\ell,j}=(n-\ell)!\ell!\sum_{k=1}^d{j-1 \choose k-1} {n-j \choose \ell-k}u^{k-\ell}.\]
Then
\[\det(A)=\prod_{k=1}^d(n-k)!k!.\]
In particular, $\det(A)\ne 0$.
\end{lemma}

\begin{proof}
First, let $\tilde{A}$ be the $d\times d$ matrix with entries
\[\tilde{A}_{\ell,j}=\sum_{k=1}^d{j-1 \choose k-1} {n-j \choose \ell-k}u^k.\]
Then
\[A_{\ell,j}=(n-\ell)!\ell!u^{-\ell}\tilde{A}_{\ell,j},\]
and so
\begin{equation}
\label{eqndetA}
\det(A)=\det(\tilde{A})\prod_{k=1}^d(n-k)!k!u^{-k}.
\end{equation}
We compute $\det(\tilde{A})$.

For $s=1,\dotsc,d$, let $\tilde{A}^{(s)}$ be the $d\times d$ matrix with entries
\[\tilde{A}_{\ell,j}^{(s)}=\begin{cases}
\sum_{k=1}^d{j-s \choose k-s} {n-j \choose \ell-k}u^k& \textrm{if }j\geq s;\\
{n-j \choose \ell-j}u^j& \textrm{if }j<s
\end{cases}\]
where the binomial coefficients are interpreted as zero in any of their entries are negative, or if the bottom entry is strictly larger than the top entry. Then $\tilde{A}^{(1)}=\tilde{A}$. We claim that for any $s$ with $1\leq s\leq d-1$, we have $\det(\tilde{A}^{(s)})=\det(\tilde{A}^{(s+1)})$.

To see the claim, we first note that for any $j$ with $s<j\leq d$ and any $k$ with $1\leq k\leq d$, we have
\[{j-s \choose k-s}={j-s-1 \choose k-s-1}+{j-s-1 \choose k-s}\]
because the top entry $j-s$ of the left hand side is nonzero. Therefore,
\begin{align*}
\tilde{A}_{\ell,j}^{(s)}&=\sum_{k=1}^d{j-s \choose k-s} {n-j \choose \ell-k}u^k\\
&=\sum_{k=1}^d\left({j-s-1 \choose k-s-1}+{j-s-1 \choose k-s} \right){n-j \choose \ell-k}u^k\\
&=\sum_{k=1}^d{j-s-1 \choose k-s-1}{n-j \choose t-k}u^k+\sum_{k=2}^{d+1}{j-s-1 \choose k-s-1} {n-j \choose \ell-k+1}u^{k-1}.
\end{align*}
The second sum above may be written instead as a sum over $k$ from $1$ to $d$ because, when $k=1$, the first binomial coefficient in the corresponding summand is zero, and when $k=d+1$, it is zero again because then $j-s-1<k-s-1$. Thus the expression above becomes
\begin{equation}
\label{eqnrowred1}
\tilde{A}_{\ell,j}^{(s)}=\tilde{A}_{\ell,j}^{(s+1)}+u^{-1}\tilde{A}_{\ell,j}^{(s+1)},
\end{equation}
when $s<j\leq d$.

When $j=s$, we have
\begin{equation}
\label{eqnrowred2}
\tilde{A}_{\ell,j}^{(s)}=\sum_{k=1}^d{0 \choose k-s} {n-j \choose \ell-k}u^k={n-j \choose \ell-j}u^j=\tilde{A}_{\ell,j}^{(s+1)}.
\end{equation}
Finally, when $1\leq j<s$, by definition we have
\begin{equation}
\label{eqnrowred3}
\tilde{A}_{\ell,j}^{(s)}={n-j \choose \ell-j}u^j=\tilde{A}_{\ell,j}^{(s+1)}.
\end{equation}

Now let $E^{(s)}$ be the elementary lower triangular $d\times d$ matrix with $1$'s along the diagonal, with entries below the diagonal $E_{j+1,j}^{(s)}=u^{-1}$ whenever $j\geq s$, and zeros everywhere else. Then putting \eqref{eqnrowred1}, \eqref{eqnrowred2} and \eqref{eqnrowred3} together imply that
\[\tilde{A}^{(s)}=E^{(s)}\tilde{A}^{(s+1)}.\]
Since obviously $\det(E^{(s)})=1$, this immediately gives the claim that $\det(\tilde{A}^{(s)})=\det(\tilde{A}^{(s+1)})$.

Thus chaining these equalities together gives $\det(\tilde{A})=\det(\tilde{A}^{(d)})$. But $\tilde{A}^{(d)}$ is upper triangular with diagonal entries $\tilde{A}_{j,j}^{(d)}=u^j$. Thus
\[\det(\tilde{A})=\det(\tilde{A}^{(d)})=\prod_{k=1}^d u^{j},\]
and combining this with \eqref{eqndetA} therefore completes the proof.
\end{proof}

We are about to proceed to the proof of Theorem \ref{thmtwists}, but we will first need to review a bit of the theory of covers. Let $P$ be a standard parabolic subgroup of $GL_n$ with Levi $M$, and let $\sigma$ be a supercuspidal representation of $M(F)$. The remarks at the beginning of this section give us a type $(K_M,\rho_M)$ for $\sigma$; it is the product of the corresponding types for each individual block of $M$. Moreover, there is a type, call it $(K,\rho)$, for $GL_n(F)$, such that the idempotent $e_\rho$ defined above acts nontrivially on an irreducible smooth representation $\pi$ of $GL_n(F)$ if and only if $\pi$ is a constituent of the parabolic induction to $GL_n(F)$ of an unramified twist of $\sigma$. Moreover, by \cite{BKgln}, one can construct $(K,\rho)$ to be a \textit{cover} of $(K_M,\rho_M)$. The notion of cover is a slightly technical one which we do not recall here; see \cite[Definition 8.1]{BKcovers}. We will only need a particular consequence of the theory of covers, which we state below.

Let $\mc{H}(\rho_M)$ and $\mc{H}(\rho)$ be the corresponding Hecke algebras for the types above; note that $\rho_M$ and $\mc{H}(\rho_M)$ decompose as tensor products according to the blocks of $M(F)$. Then we have the following.

\begin{theorem}[Bushnell--Kutzko]
\label{thmcovers}
In the setting above, there is a linear map $t:\mc{H}(\rho_M)\to\mc{H}(\rho)$ such that for any irreducible smooth representation $\pi$ of $GL_n(F)$, we have
\[\tr(t(f)|\pi)=\tr(f|\Jac_M(\pi)),\]
for any $f\in\mc{H}(\rho_M)$.
\end{theorem}

\begin{proof}
This follows from \cite[Proposition 7.9, Corollary 7.12]{BKcovers}.
\end{proof}

To prove Theorem \ref{thmepsilon}, we first make a Hecke theoretic construction of epsilon factors in the case of \textit{simple types}: Let $\tau$ be a supercuspidal representation of $GL_{n_\tau}(F)$ for some $n_\tau$, and let $\Delta_1,\dotsc,\Delta_r$ be segments of the form
\[\Delta_i=\{\tau,\tau\otimes\vert\det\vert,\dotsc,\tau\otimes\vert\det\vert^{\ell_i-1}\},\]
with $\ell_i$ the length of $\Delta_i$. Let $n=\sum_i\ell_i$. Let $P\subset GL_{nn_\tau}$ be the standard parabolic subgroup with Levi $M=GL_{n_\tau}\times\dotsb\times GL_{n_\tau}$ ($n$ times). Let $(\rho_\tau,K_\tau)$ be a type for $GL_{n_\tau}(F)$ for which the comments after Definition \ref{deftwind} hold; in particular the associated Hecke algebra $\mc{H}(\rho_\tau)$ is a commutative integral domain and satisfies the conclusions of Lemmas \ref{lemheckehomog} and \ref{lemtwindideal}. Let $K_{M}=K_\tau\times\dotsb\times K_\tau$ and $\rho_{M}=\rho_\tau\boxtimes\dotsb\boxtimes\rho_\tau$, so that $(K_{M},\rho_{M})$ is a type for the representation $\tau\boxtimes\dotsb\boxtimes\tau$ of $M(F)$. Let $(K,\rho)$ be a type for $GL_{nn_\tau}(F)$ covering $(K_M,\rho_M)$, and let $\mc{H}(\rho)$ be the associated Hecke algebra. We show the following.

\begin{lemma}
\label{lemredtosimple}
Let $\Delta_1,\dotsc,\Delta_r$ be, as above, segments each beginning with the same supercuspidal representation $\tau$ of $GL_{n_\tau}(F)$. Write, for any $r$-tuple of nonzero complex numbers $z=(z_1,\dotsc,z_r)\in(\C^\times)^r$,
\[\pi(z)=\Ind_{P(F)}^{GL_{nn_\tau}(F)}((Q(\Delta_1)\otimes\chi^{(z_1)})\boxtimes\dotsb\boxtimes(Q(\Delta_r)\otimes\chi^{(z_r)})),\]
with $P$ the appropriate parabolic.

Fix a positive integer $\ell$. Let $n(\ell)$ be the number of segments $\Delta_i$ which are of length $\ell$, and let $i_1,\dotsc,i_{n(\tau,\ell)}$ be the corresponding indices $i$. Let $S$ be any symmetric polynomial in $n(\ell)$ variables. Then there are $f_1,\dotsc,f_m\in\mc{H}(\rho)$ (where $\mc{H}(\rho)$ is as above) for some positive integer $m$, and a polynomial $Q$ in $m$ variables, such that
\[Q(\tr(f_1|\pi(z)),\dotsc,\tr(f_m|\pi(z)))=S(z_{i_1}^{I(\tau)},\dotsc,z_{i_{n(\ell)}}^{I(\tau)}),\]
where, as usual, $I(\tau)$ denotes the twisting index of $\tau$.
\end{lemma}

\begin{proof}
Let $t$ be the map from Theorem \ref{thmcovers} in the context above. Then
\begin{equation}
\label{eqntrtjac}
\tr(t(f)|\pi(z))=\tr(f|\Jac_M(\pi(z))),\qquad f\in\mc{H}(\rho_\tau).
\end{equation}
Let $d$ be any integer and let $f_0\in\mc{H}(\rho_\tau)_d$ (see Lemma \ref{lemheckehomog}). Let $\mc{H}(\rho_M)$ be the Hecke algebra for $\rho_M$, so
\[\mc{H}(\rho_M)=\mc{H}(\rho_\tau)\otimes\dotsb\otimes\mc{H}(\rho_\tau)\qquad(n\textrm{ times}).\]
Let $e_{\rho_\tau}\in\mc{H}(\rho_\tau)$ be the idempotent for $\rho_\tau$. For any $j=1,\dotsc,n$, let $f_j\in\mc{H}(\rho_M)$ be the operator given by
\[f_j=e_{\rho_\tau}\otimes\dotsb\otimes e_{\rho_\tau}\otimes f_0\otimes e_{\rho_\tau}\otimes\dotsb\otimes e_{\rho_\tau}.\]
Finally, let $d_{\tau}=\dim_{\C}(\tau[\rho_\tau])$ (which equals $\dim(\rho_\tau)$). Then Lemma \ref{lemjacisotypic} implies
\begin{equation}
\label{eqntrjac1}
\tr(f_j|\Jac_M(\pi(z)))=\sum_{i=1}^r\sum_{k=1}^n(d_{\tau}^{n-1}m(\ell_i,k))\tr(f_0|\tau)z_i^d q^{k-\ell_i-\frac{n+1}{2}+j},
\end{equation}
where $q$ is the order of the residue field of $F$ and
\[m(\ell_i,k)=(n-\ell_i)!\ell_i!{j-1 \choose k-1} {n-j \choose \ell_i-k}\prod_{i'=1}^{r}\frac{1}{\ell_{i'}!}.\]

Now let $A$ be the $n\times n$ matrix with entries
\[A_{\ell,j}=(n-\ell)!\ell!\sum_{k=1}^d{j-1 \choose k-1} {n-j \choose \ell-k}q^{k-\ell}.\]
Let $f_j'\in\mc{H}(\rho_M)$ be given by
\[\left(\prod_{i'=1}^{r}\frac{1}{\ell_{i'}!}q^{\frac{n+1}{2}-j}\right)f_j,\]
Let $T$ be the column vector of length $n$ with $j$th entry
\[T_j=\tr(f_j'|\Jac_M(\pi(z)))\]
and $U$ the column vector of length $n$ with $\ell$th entry
\[U_\ell=\sum_{\substack{1\leq i\leq r\\ \ell_i=\ell}}\tr(f_0|\tau)z_i^d.\]
Then \eqref{eqntrjac1} simply reads
\[T=AU.\]
By Lemma \ref{lemdeterminant}, $A$ is invertible, hence
\[A^{-1}T=U,\]
which means that for any $\ell=1,\dotsc,n$, there is a linear combination of the Hecke operators $f_j'$ whose trace on $\Jac_M(\pi(z))$ equals $U_\ell$.

Now if $d$ is divisible by the twisting index $I(\tau)$ of $\tau$, then by Lemma \ref{lemtwindideal}, we may choose our $f_0\in\mc{H}(\rho_\tau)_d$ so that $\tr(f_0|\tau)=1$. Thus with this choice, we have
\[U_\ell=\sum_{\substack{1\leq i\leq r\\ \ell_i=\ell}}z_i^d.\]
By the theory of symmetric polynomials, the power sum polynomials generate the whole ring of symmetric polynomials. The conclusion of the lemma follows immediately.
\end{proof}

\begin{theorem}
\label{thmtwists}
Let $\Delta_1,\dotsc,\Delta_r$ be segments, and write, for any $r$-tuple of nonzero complex numbers $z=(z_1,\dotsc,z_r)\in(\C^\times)^r$,
\[\pi(z)=\Ind_{P(F)}^{GL_{n}(F)}((Q(\Delta_1)\otimes\chi^{(z_1)})\boxtimes\dotsb\boxtimes(Q(\Delta_r)\otimes\chi^{(z_r)})),\]
with $P$ the appropriate parabolic.

Fix a positive integer $\ell$ and a supercuspidal representation $\tau$. Let $n(\tau,\ell)$ be the number of segments $\Delta_k$ which are of length $\ell$ and contain an unramified twist of $\tau$, and let $k_1,\dotsc,k_{n(\tau,\ell)}$ be the corresponding indices $k$. For any such $k\in\{k_1,\dotsc,k_{n(\tau,\ell)}\}$, let $w_k\in\C^\times$ be a nonzero complex number such that $\Delta_k$ begins with $\tau\otimes\chi^{(w_k)}$. Finally, let $S$ be any symmetric polynomial in $n(\tau,\ell)$ variables, and denote by $\mc{H}$ the Hecke algebra for $GL_n(F)$. Then there are $f_1,\dotsc,f_m\in\mc{H}$ for some positive integer $m$, and a polynomial $Q$ in $m$ variables, such that for any $z=(z_1,\dotsc,z_r)\in(\C^\times)^r$, we have
\[Q(\tr(f_1|\pi(z)),\dotsc,\tr(f_m|\pi(z)))=S((w_{k_1}z_{k_1})^{I(\tau)},\dotsc,((w_{k_{n(\tau,\ell)}}z_{k_{n(\tau,\ell)}})^{I(\tau)}),\]
where, as usual, $I(\tau)$ denotes the twisting index of $\tau$.
\end{theorem}

\begin{proof}
We begin by noting that the order of the segments in the definition of $\pi(z)$ does not matter since we are only concerned with taking traces. Therefore, we can and will assume that there is a positive integer $r'\leq r$ and indices $k_0,k_1,\dotsc,k_{r'}$ with $k_0=0$ and $k_{r'}=r$, and with $k_0\leq\dotsb\leq k_{r'}$, such that for all $j=1,\dotsc,r'$, the segments $\Delta_{k_{j-1}+1},\Delta_{k_{j-1}+2},\dotsc,\Delta_{k_j}$ contain only twists of the same supercuspidal representation, say $\tau_j$, and that all the $\tau_j$'s are distinct even up to unramified twist. In other words, we group all segments with unramified twists of the same supercuspidal together. 

For each $j=1\dotsc,r'$, let $n_j$, $P^{(j)}$, and $\widetilde{P}$ be, respectively, the positive integer, the standard parabolic subgroup of $GL_{n_j}$, and the standard parabolic subgroup of $GL_n$, such that
\[\pi(z)=\Ind_{\widetilde{P}(F)}^{GL_n(F)}(\pi_z^{(1)}\boxtimes\dotsb\boxtimes\pi_z^{(r')}),\]
for all $z=(z_1,\dotsc,z_r)\in(\C^\times)^r$, where
\[\pi_z^{(j)}=\Ind_{P^{(j)}(F)}^{GL_n(F)}(Q(\Delta_{k_{j-1}+1}\otimes\chi^{(z_{k_{j-1}+1})})\boxtimes\dotsb\boxtimes Q(\Delta_{k_j}\otimes\chi^{(z_{k_j})})).\]
Then the Levi of $\widetilde{P}$, call it $\widetilde{M}$, is equal to $GL_{n_1}\times\dotsb\times GL_{n_{r'}}$.

Let $(K_\rho,\rho)$ be a type for $GL_n(F)$ which covers the type for the supercuspidal representation of the appropriate Levi which supports any of the $\pi(z)$'s. Then for each $j=1\dotsc,r'$, let $K_j=K_\rho\cap GL_{n_j}(F)$ (where $GL_{n_j}(F)$ is viewed as a factor of $\widetilde{M}(F)$) and let $\rho_j=\rho|_{K_j}$. Then $(K_j,\rho_j)$ is a type for $GL_{n_j}(F)$ which covers the type for the supercuspidal representation of the appropriate Levi which supports $\pi_z^{(j)}$, namely $\tau_j\boxtimes\dotsb\boxtimes\tau_j$; this follows either directly from the construction of $(K_\rho,\rho)$ in \cite{BKgln} or from the general theory of covers \cite[Proposition 8.5 (ii)]{BKcovers}. Let $e_j$ be the idempotent associated with $\rho_j$, let $\mc{H}(\rho_j)$ the Hecke algebra associated with $\rho$, and let $\mc{H}(\rho)$ be that associated with $\rho$.

Now fix $j$ with $1\leq j\leq r'$, and let $f\in\mc{H}(\rho_{j})$ be any operator. Let
\[t:\mc{H}(\rho_1)\otimes\dotsb\otimes\mc{H}(\rho_{r'})\to\mc{H}(\rho)\]
be the map given by Theorem \ref{thmcovers} applied with $\widetilde{M}$ as the Levi. Let $\tilde{f}$ be the operator
\[\tilde{f}=e_1\otimes\dotsb\otimes e_{j-1}\otimes f\otimes e_{j+1}\otimes\dotsb\otimes e_{r'}\in\mc{H}(\rho_1)\otimes\dotsb\otimes\mc{H}(\rho_{r'}).\]
Then
\begin{equation}
\label{eqntoff}
\tr(t(\tilde{f})|\pi(z))=\tr(\tilde{f}|\Jac_{\widetilde{M}}(\pi(z))),
\end{equation}
for any $z$. We claim that
\begin{multline}
\label{eqntildetrjac}
\tr(\tilde{f}|\Jac_{\widetilde{M}}(\pi(z)))\\
=\tr(f|\pi_z^{(j)}\otimes(\delta_{\widetilde{P}}^{-1/2}|_{GL_{n_{j}}(F)}))\prod_{\substack{1\leq j'\leq r'\\ j'\ne j}}\tr(e_{j'}|\pi_z^{(j')}\otimes(\delta_{\widetilde{P}}^{-1/2}|_{GL_{n_{j'}}(F)})).
\end{multline}
This will follow from Proposition \ref{propcasselman} with both parabolics there being $\widetilde{P}$, as we shall see now.

Assume $w\in W_n$ is an element with $w\notin W_{\widetilde{M}}$. Then
\[\Jac_{M(F)\cap wM(F)w^{-1}}(\pi(z))\]
is a representation of $M(F)\cap wM(F)w^{-1}$ which, if nonzero, has the same supercuspidal support at $\pi(z)$. Since $w\notin W_{\widetilde{M}}$, the representation
\[w^{-1}(\Jac_{M(F)\cap wM(F)w^{-1}}(\pi(z)))\]
has different supercuspidal support than $\pi(z)$ if it is nonzero, and therefore $\mc{H}(\rho_1)\otimes\dotsb\otimes\mc{H}(\rho_{r'})$ acts by zero on
\[\mathrm{unInd}_{w^{-1}\widetilde{P}(F)w\cap \widetilde{M}(F)}^{\widetilde{M}(F)}(w^{-1}(\Jac_{\widetilde{M}(F)\cap w\widetilde{M}(F)w^{-1}}(\pi(z)))\delta_w^{1/2}),\]
where $\delta_w^{1/2}$ is as in Proposition \ref{propcasselman}. It therefore follows from that proposition that
\[\tr(\tilde{f}|\Jac_{\widetilde{M}(F)}(\pi(z)))=\tr(\tilde{f}|(\pi_z^{(1)}\boxtimes\dotsb\boxtimes\pi_z^{(r')})\otimes\delta_{\widetilde{P}}^{-1/2}),\]
since the only constituent of $\Jac_{\widetilde{M}}(\pi(z))$ on which $\mc{H}(\rho_1)\otimes\dotsb\otimes\mc{H}(\rho_{r'})$ acts nontrivially is the one corresponding to $w=1$ there. Since $\delta_w=\delta_{\widetilde{P}}^{-1}$ for $w=1$, the claim immediately follows.

Now for any $j'=1,\dotsc,r'$, we have that $\tr(e_{j'}|\pi_z^{(j')}\otimes(\delta_{\widetilde{P}}^{-1/2}|_{GL_{n_{j'}}(F)}))$ is the dimension of the $\rho_{j'}$-isotypic component of $\pi_z^{(j')}\otimes(\delta_{\widetilde{P}}^{-1/2}|_{GL_{n_{j'}}(F)})$, and this is the same as the dimension of the $\rho_{j'}$-isotypic component of $\pi_z^{(j')}$ by Lemma \ref{lemunrtwist}. Thus \eqref{eqntildetrjac} becomes
\begin{equation}
\label{eqntildetrjac2}
\tr(\tilde{f}|\Jac_{\widetilde{M}}(\pi(z)))=\tr(f|\pi_z^{(j)}\otimes(\delta_{\widetilde{P}}^{-1/2}|_{GL_{n_{j}}(F)}))\prod_{\substack{1\leq j'\leq r'\\ j'\ne j}}\tr(e_{j'}|\pi_z^{(j')}).
\end{equation}
For any $j=1,\dotsc,r'$, write
\[a_j=\prod_{\substack{1\leq j'\leq r'\\ j'\ne j}}\tr(e_{j'}|\pi_z^{(j')}),\]
which is a nonzero integer, and let $b_j\in\C^\times$ be the nonzero complex number such that $\delta_{\widetilde{P}}^{-1/2}|_{GL_{n_j}(F)}=\chi^{(b_j)}\circ\det$. Then \eqref{eqntoff} and \eqref{eqntildetrjac2} give
\begin{equation}
\label{eqntrtfistrf}
a_{j}^{-1}\tr(t(\tilde{f})|\pi(z))=\tr(f|\pi_z^{(j)}\otimes\chi^{(b_{j})}).
\end{equation}

Now let $S_j$ be a symmetric polynomial in $k_j-k_{j-1}$ variables, and let $S_j'$ be the symmetric polynomial defined by
\[S_j'(X_1,\dotsc,X_{k_j-k_{j-1}})=S_j(b_j^{-I(\tau)}X_1,\dotsc,b_j^{-I(\tau)}X_{k_j-k_{j-1}}).\]
For $i\in\{k_{j-1}+1,\dotsc,k_j\}$, let $w_i\in\C^\times$ be such that $\Delta_i$ begins with $\tau_j\otimes\chi^{(w_i)}$. Then by Lemma \ref{lemredtosimple} (applied with $S_j'$) we can find operators $f_1^{(j)},\dotsc,f_{m_j}^{(j)}\in\mc{H}(\rho_j)$ for some positive integer $m_j$, and a polynomial $Q_j$ in $m_j$ variables, such that for any $z=(z_1,\dotsc,z_r)\in(\C^\times)^r$, we have
\begin{multline*}
Q_j(\tr(f_1^{(j)}|\pi_z^{(j)}\otimes\chi^{(b_{j})}),\dotsc,\tr(f_{m_j}^{(j)}|\pi_z^{(j)}\otimes\chi^{(b_{j})}))\\
\begin{aligned}
&{}=S_j'((b_j w_{k_{j-1}+1}z_{k_{j-1}+1})^{I(\tau)},\dotsc,(b_j w_{k_j}z_{k_j})^{I(\tau)})\\
&{}=S_j((w_{k_{j-1}+1}z_{k_{j-1}+1})^{I(\tau)},\dotsc,(w_{k_j}z_{k_j})^{I(\tau)}).
\end{aligned}
\end{multline*}
For $l=1,\dotsc,m_j$, write
\[\tilde{f}_l^{(j)}=e_1\otimes\dotsc\otimes e_{j-1}\otimes f_l^{(j)}\otimes e_{j+1}\otimes\dotsc\otimes e_{r'}\in\mc{H}(\rho_1)\otimes\dotsb\otimes\mc{H}(\rho_{r'}).\]
Then by \eqref{eqntrtfistrf} we have
\begin{equation*}
a_j^{-1}Q_j(\tr(t(\tilde{f}_1^{(j)})|\pi(z)),\dotsc,\tr(t(\tilde{f}_{m_j}^{(j)})|\pi(z)))
=S_j((w_{k_{j-1}+1}z_{k_{j-1}+1})^{I(\tau)},\dotsc,(w_{k_j}z_{k_j})^{I(\tau)}).
\end{equation*}
This completes the proof of the theorem.
\end{proof}

\section{Unlinkedness}
\label{secunlinked}
In this section we explain a mild technical hypothesis which we need in order to prove our coherence results on epsilon factors in Section \ref{secepsilon}.

\begin{definition}
Let $F$ be a nonarchimedean local field of characteristic zero and $\pi$ an irreducible smooth representation of $GL_n(F)$. Write $\pi\cong Q(\Delta_1,\dotsc,\Delta_r)$ with $\Delta_1,\dotsc,\Delta_r$ segments as usual. We say $\pi$ is \textit{unlinked} if no two of the segments $\Delta_i,\Delta_j$, $1\leq i,j\leq r$, are linked.

If instead $F_1,\dotsc,F_N$ are nonarchimedean local fields of characteristic zero and $\pi=\pi_1\boxtimes\dotsb\boxtimes\pi_N$ is a representation of $\prod_i GL_{n_i}(F_i)$ with each $\pi_i$ irreducible and smooth, then we say $\pi$ is \textit{unlinked} if each $\pi_i$ is.
\end{definition}

Essentially tempered representations are unlinked (we wait until the end of this section to prove this). The utility of this definition lies in the fact that if $\pi\cong Q(\Delta_1,\dotsc,\Delta_r)$ is unlinked, then $\pi$ is isomorphic to the whole parabolically induced representation $\pi(\Delta_1,\dotsc,\Delta_r)$, which is thus irreducible. We will use this fact several times in this section.

For the rest of this section we fix nonarchimedean local fields $F_1,\dotsc,F_N$ of characteristic zero and positive integers $n_1,\dotsc, n_N$. Write, as before, for each $i=1,\dotsc,N$,
\[\mc{H}_i=C_c^\infty(GL_{n_i}(F_i),\C),\]
and we also write as before,
\[\mc{H}=\mc{H}_1\otimes_\C\dotsb\otimes_\C\mc{H}_N.\]

\begin{proposition}
\label{propunlinked}
Let $\mc{F}=(X,\Sigma,R,T)$ be a generically irreducible family of smooth admissible representations of $\prod_i GL_{n_i}(F_i)$. For any $x\in\Sigma$, let $\pi_x$ be the representation of $\prod_i GL_{n_i}(F_i)$ defined by
\[T_x(f)=\tr(f|\pi_x),\qquad f\in\mc{H}.\]
Assume that for some $x_0\in\Sigma$, the representation $\pi_{x_0}$ is unlinked. Then there is an open and closed neighborhood $X_0$ of $x_0$ in $X$ such that for any $x\in\Sigma\cap X_0$, the representation $\pi_x$ is unlinked as well.
\end{proposition}

\begin{proof}
For any $x\in\Sigma$ and any $i=1,\dotsc,N$, let $\pi_{x,i}$ be the representation of $GL_{n_i}(F_i)$ so that
\[\pi_x\cong\pi_{1,x}\boxtimes\dotsb\boxtimes\pi_{N,x}.\]
Since for any $i$ the representation $\pi_{i,x_0}$ is unlinked, it can be written as an irreducible parabolic induction
\[\pi_{i,x_0}\cong\Ind_{P(F_i)}^{GL_{n_i}(F_i)}(Q(\Delta_{1,i})\boxtimes\dotsb\boxtimes Q(\Delta_{r,i}))\]
for some segments $\Delta_{k,i}$, where $P$ is the appropriate standard parabolic subgroup of $GL_{n_i}$. By Theorem \ref{thmfamilythm}, there is an open and closed neighborhood of $x_0$, say $X_0'$, such that for any $i$ and any $x\in\Sigma\cap X_0'$, we have
\begin{equation*}
\pi_{i,x}\cong Q((\Delta_{1,i}\otimes\chi_{1,i,x}),\dotsc,(\Delta_{r,i}\otimes\chi_{r,i,x}))
\end{equation*}
for some unramified characters $\chi_{k,i,x}$ of $F_i^\times$.

Now given any choice of unramified characters $\chi_{1,i},\dotsc,\chi_{r,i}$ of $F_i^\times$, denote by $\mc{S}(\chi_{1,i},\dotsc,\chi_{r,i})$ the poset of multisets of segments obtained from $\{(\Delta_{1,i}\otimes\chi_{1,i}),\dotsc,(\Delta_{r,i}\otimes\chi_{r,i})\}$ by elementary operations. On the one hand, Zelevinky's character formula \cite[Proposition 9.13]{zel} lets us write, virtually,
\begin{equation}
\label{eqnzelfmla}
Q((\Delta_{1,i}\otimes\chi_{1,i}),\dotsc,(\Delta_{r,i}\otimes\chi_{r,i}))=\sum_{S\in\mc{S}(\chi_{1,i},\dotsc,\chi_{r,i})}a_{S}\pi(S),
\end{equation}
where the constants $a_{S}$ are integers which only depend on the supercuspidals in $\Delta_{1,i},\dotsc,\Delta_{r,i}$ up to unramified twist, on the poset structure of $\mc{S}(\chi_{1,i},\dotsc,\chi_{r,i})$, and on the lengths of the segments of each multiset $S\in\mc{S}(\chi_{1,i},\dotsc,\chi_{r,i})$. Clearly there are only finitely many possibilities for this data as $\chi_{1,i},\dotsc,\chi_{r,i}$ range over all unramified characters of $F_i^\times$.

On the other hand, by Lemma \ref{lemunrtwist}, for any open normal subgroup $K_i\subset GL_{n_i}(\mc{O}_{F_i})$ and any multiset of segments $S'$ for $GL_{n_i}(F_i)$, we have that
\[\dim_\C\pi(S')^{K_i}\]
depends on $S'$ only up to unramified twisting of the segments in $S'$. It therefore follows from \eqref{eqnzelfmla} that the set
\[\sset{\dim_\C Q(S)^{K_i}}{S\in\mc{S}(\chi_{1,i},\dotsc,\chi_{r,i})\textrm{ for some }\chi_{1,i},\dotsc,\chi_{r,i}}\]
is finite, for any fixed $K_i$. So choose, for each $i$, such a $K_i$ small enough so that $\pi(S)^{K_i}$, $Q(S)^{K_i}$, and $\ker(\pi(S)\to Q(S))^{K_i}$ (where $\pi(S)\to Q(S)$ is the natural quotient map) are all nonzero, for any choice of unramified characters $\chi_{k,i}$ of $F_i^\times$ and any $S\in\mc{S}(\chi_{1,i},\dotsc,\chi_{r,i})$. Then in particular, for any such $S$, we have
\[\dim_\C\pi(S)^{K_i}<\dim_\C Q(S)^{K_i}.\] 

Now let $K=\prod_i K_i$ and consider the function $T(\chars(K))\in R$. This can only take finitely many values on $\Sigma\cap X_0'$ by what we just showed, and hence is locally constant on $X_0'$. Let $X_0$ be an open and closed neighborhood of $x_0$ in $X_0$ where $T(\chars(K))$ is constant. If for any $i$ we had that $\pi_{i,x}$ is not unlinked for some $x\in\Sigma\cap X_0$, then by the discussion above we would have
\[\tr(\chars(K_i)|\pi_{i,x})<\tr(\chars(K_i)|\pi_{i,x_0}).\]
Whence $T_x(\chars(K))<T_{x_0}(\chars(K))$, a contradiction. Therefore $\pi_x$ is unlinked for every $x\in \Sigma\cap X_0$, and this completes the proof.
\end{proof}

We now prove the following proposition which will allow us to reduce to the case where we deal with a family for only one group $GL_n(F)$.

\begin{proposition}
\label{propredtoonegln}
Let $\mc{F}=(X,\Sigma,R,T)$ be a generically irreducible family of smooth admissible representations of $\prod_i GL_{n_i}(F_i)$. For any $x\in\Sigma$ and any $i=1,\dotsc,N$, let $\pi_{i,x}$ be the representation of $GL_{n_i}(F_i)$ so that
\[T_x(f)=\tr(f|\pi_{1,x}\boxtimes\dotsb\boxtimes\pi_{N,x}),\]
for any $f\in\mc{H}$. Assume that for some $x_0\in\Sigma$, the representations $\pi_{i,x_0}$ are unlinked for every $i$. Then there is an open and closed neighborhood $X_0$ of $x_0$ in $X$ with the following property: Let $R_0$ be the ring given by
\[R_0=\sset{\phi|_{X_0}}{\phi\in R}.\]
Then for each $i=1,\dotsc,N$, there is a linear map $T_i:\mc{H}_i\to R_0$ such that
\[T_{i,x}(f_i)=\tr(f_i|\pi_{i,x})\]
for every $x\in\Sigma\cap X_0$ and $f_i\in\mc{H}_i$, and such that $\mc{F}_i=(X_0,\Sigma\cap X_0,R_0,T_i)$ is a generically irreducible family smooth admissible representations of $GL_{n_i}(F_i)$.
\end{proposition}

\begin{proof}
For any $i=1,\dotsc,N$, because of the hypothesis that $\pi_{i,x_0}$ is unlinked, we have
\[\pi_{i,x_0}\cong\Ind_{P(F_i)}^{GL_{n_i}(F_i)}(Q(\Delta_{1,i})\boxtimes\dotsb\boxtimes Q(\Delta_{r_i,i}))\]
for some segments $\Delta_{k,i}$, where $P$ is the appropriate standard parabolic subgroup of $GL_{n_i}$. Then Theorem \ref{thmfamilythm} and Proposition \ref{propunlinked} imply the existence of an open and closed neighborhood $X_i$ of $x_0$ in $X$ such that for any $x\in\Sigma\cap X_i$, we have
\begin{equation}
\label{eqindreptwistsi}
\pi_{i,x}\cong\Ind_{P(F_i)}^{GL_{n_i}(F_i)}(Q(\Delta_{1,i}\otimes\chi_{1,i,x})\boxtimes\dotsb\boxtimes Q(\Delta_{r_i,i}\otimes\chi_{r_i,i,x}))
\end{equation}
for some unramified characters $\chi_{1,i,x},\dotsc,\chi_{r_i,i,x}$ of $F_i^\times$. We take $X_0$ to be the intersection of all the $X_i$'s.

Now fix $i$. For any $j=1,\dotsc,N$ with $j\ne i$, let $K_j'$ be an open normal subgroup of $GL_{n_j}(\O_{F_j})$ for which $\pi_{j,x_0}$ has nonzero fixed vectors by $K_j'$. Then Lemma \ref{lemunrtwist} and \eqref{eqindreptwistsi} imply that for any $x\in\Sigma\cap X_0$,
\[\dim_{\C}(\pi_{j,x}^{K_j'})=\dim_{\C}(\pi_{j,x_0}^{K_j'}).\]
So if we let $e_j\in\mc{H}_j$ be the operator
\[e_j=\vol(K_j')^{-1}(\dim_{\C}(\pi_{j,x_0}^{K_i'}))^{-1}\chars(K_j'),\]
then for any $x\in\Sigma\cap X_0$, we have
\[\tr(e_j|\pi_{j,x})=1.\]
We then define $T_i$ by
\[T_i(f_i)=T(e_1\otimes\dotsb\otimes e_{i-1}\otimes f_i\otimes e_{i+1}\otimes\dotsb\otimes e_N)|_{X_0}.\]
for $f_i\in\mc{H}_i$. Then it is easy to check that $T_i$ has the properties stated in the proposition.
\end{proof}

As promised earlier, we have the following proposition.

\begin{proposition}
\label{proptempunl}
Let $F$ be a nonarchimedean local field of characteristic zero, and let $\pi$ be an irreducible smooth representation of $GL_n(F)$. If $\pi$ is essentially tempered, then $\pi$ is unlinked.
\end{proposition}

\begin{proof}
Given an integer $\ell\geq 1$ and supercuspidal representation $\tau$ of $GL_{n_\tau}(F)$ for some $n_\tau$, let us write $\Delta(\tau,\ell)$ for the segment
\[\Delta(\tau,\ell)=\{\tau\otimes\vert\det\vert^{\frac{1-\ell}{2}},\tau\otimes\vert\det\vert^{\frac{3-\ell}{2}},\dotsc,\tau\otimes\vert\det\vert^{\frac{\ell-1}{2}}\}.\]
Then the tempered representations of $GL_n(F)$ are those of the form
\[Q(\Delta(\tau_1,\ell_1),\dotsc,\Delta(\tau_r,\ell_r)),\]
where the $\ell_i$'s are any positive integers and the $\tau_i$'s are \textit{unitary} supercuspidal representations of some groups $GL_{n_i}(F)$, and $\sum_i n_i\ell_i=n$. Thus the essentially tempered ones are those of the form
\[Q(\Delta(\tau_1,\ell_1)\otimes\chi,\dotsc,\Delta(\tau_r,\ell_r)\otimes\chi),\]
where the $\tau_i$'s and $\ell_i$'s are as just described and $\chi$ is an unramified character of $F^\times$. Since the $\tau_i$'s are unitary, given $i,j$ with $1\leq i,j\leq r$, the only way $\Delta(\tau_i,\ell_i)\otimes\chi$ and $\Delta(\tau_j,\ell_j)\otimes\chi$ can overlap is if $\tau_i=\tau_j$. But then in this case one of these segments must be contained in the other, and therefore they are not linked by definition. This implies the proposition.
\end{proof}

\section{Variation of epsilon factors}
\label{secepsilon}
Let $F$ be a nonarchimedean local field of characteristic zero and $n$ a nonnegative integer. We begin by discussing epsilon factors for Weil--Deligne representations (see, for instance, \cite{tatecor} for necessary background). Below we will use the local Langlands correspondence to define epsilon factors $\epsilon(\pi,s,\psi,\varrho)$ for irreducible smooth representations $\pi$ of $GL_n(F)$ attached to the data of a complex number $s\in\C$, a nontrivial continuous character $\psi:F\to\C^\times$, and a finite dimensional algebraic representation $\varrho$ of $GL_n(\C)$. When $\varrho$ is the standard representation, of course this definition agrees with the usual one tautologically because of the characterizing properties of the local Langlands correspondence.

So fix a uniformizer $\varpi\in F$. Write $W_F$ for the Weil group of $F$, and fix a geometric Frobenius element $\Phi$ in $W_F$. We normalize the reciprocity map $\rec:F^\times\to W_F$ of local class field theory to send $\varpi$ to $\Phi$. Recall that for $z\in\C^\times$, we have written $\chi^{(z)}$ for the unramified character of $F^\times$ for which $\chi^{(z)}(\varpi)=z$. If $R=(r,N)$ is a Weil--Deligne representation for $F$, we abusively write $R\otimes\chi^{(z)}$ for the twisted representation $(r\otimes(\chi^{(z)}\circ\rec),N\otimes 1)$.

\begin{definition}
\label{defepsilon}
Fix a positive integer $M$ and let $n^{(1)},\dotsc,n^{(M)}$ be nonnegative integers. For $j=1,\dotsc,M$, let $\pi^{(j)}$ be an irreducible smooth representation of $GL_{n^{(j)}}(F)$, and let $R_{\pi^{(j)}}$ be the corresponding Weil--Deligne representation under the local Langlands correspondence of Harris--Taylor and Henniart. Fix a complex number $s\in\C$, a nontrivial continuous character $\psi:F\to\C^\times$, and for each $j=1,\dotsc,M$, a finite dimensional algebraic representation $\varrho^{(j)}$ of $GL_{n^{(j)}}(\C)$. We define the \textit{epsilon factor} attached to this data to be
\[\epsilon(\boxtimes_{j=1}^M\pi^{(j)},s,\psi,\boxtimes_{j=1}^M\varrho^{(j)})=\epsilon(\otimes_{j=1}^M(\varrho^{(j)}\circ R_{\pi^{(j)}}),s,\psi),\]
where the right hand side denotes the usual epsilon factor attached to a Weil--Deligne representation.
\end{definition}

We will need the following properties of epsilon factors.

\begin{lemma}
\label{lemepstwists}
Fix a nontrivial additive character $\psi$ of $F$ and let $s\in\C$.
\begin{enumerate}[label=(\alph*)]
\item Let $R_1,R_2$ be Weil--Deligne representations for $F$. Then
\[\epsilon(R_1\oplus R_2,s,\psi)=\epsilon(R_1,s,\psi)\epsilon(R_2,s,\psi)\]
\item Let $r$ be an irreducible representation of $W_F$ of dimension $n$, let $m\geq 1$ be an integer, and let $z\in\C^\times$. Then
\[\epsilon(r\otimes\Sp(m)\otimes\chi^{(z)},s,\psi)=z^{m(a(r)+\cond(\psi)n)+\dim_\C(r^{I_F})}\epsilon(r\otimes\Sp(m),s,\psi),\]
where $\Sp(m)$ is the special representation of dimension $m$, $\cond(\psi)$ is the conductor of $\psi$, the integer $a(r)$ is the Artin conductor of $r$, and $r^{I_F}$ denotes the invariants of $r$ under the inertia group $I_F$ of $F$.
\end{enumerate}
\end{lemma}

\begin{proof}
This is standard; see for instance \cite[\S 3.4]{tatecor}.
\end{proof}

The following is our main theorem on epsilon factors.

\begin{theorem}
\label{thmepsilon}
Let $N,M$ be a positive integers. For each $i=1,\dotsc,N$ and each $j=1,\dotsc,M$, fix the following:
\begin{itemize}
\item Nonnegative integers $n_i^{(j)}$;
\item Nonarchimedean local fields $F_i$ of characteristic zero;
\item Generically irreducible families $\mc{F}^{(j)}=(X^{(j)},\Sigma^{(j)},R^{(j)},T^{(j)})$ of smooth admissible representations of $\prod_i GL_{n_i^{(j)}}(F_i)$;
\item Complex numbers $s_i\in\C$;
\item Nontrivial additive characters $\psi_i$ of $F_i$;
\item Finite dimensional algebraic representations $\varrho_i^{(j)}$ of $GL_{n_i^{(j)}}(\C)$.
\end{itemize}
Write $\mc{H}^{(j)}$ for the Hecke algebra of $\prod_i GL_{n_i^{(j)}}(F_i)$, and for any $x^{(j)}\in\Sigma^{(j)}$, let $\pi_{i,x^{(j)}}^{(j)}$ be the representation of $GL_{n_i^{(j)}}(F_i)$ so that
\[T_{x^{(j)}}^{(j)}(f)=\tr(f|\pi_{1,x^{(j)}}^{(j)}\boxtimes\dotsb\boxtimes\pi_{N,x^{(j)}}^{(j)}),\]
for any $f\in\mc{H}^{(j)}$. Finally, for each $j$, assume that for some point $x_0^{(j)}\in\Sigma^{(j)}$, the representations $\pi_{i,x_0}^{(j)}$ are unlinked for every $i$. Then there are functions $\mc{E}_i^{(j)}\in R^{(j)}$ and open and closed neighborhoods $X_0^{(j)}$ of $x_0^{(j)}$ in $X^{(j)}$ such that for any choices of $x^{(j)}\in\Sigma^{(j)}\cap X_0^{(j)}$ and any $i$, we have
\[\prod_{j=1}^M\mc{E}_i^{(j)}(x^{(j)})=\epsilon(\boxtimes_{j=1}^M\pi_{i,x^{(j)}}^{(j)},s_i,\psi_i,\boxtimes_{j=1}^M\varrho^{(j)}).\]
\end{theorem}

\begin{proof}
By Proposition \ref{propunlinked}, the unlinkedness hypothesis gives us, for each $j=1,\dotsc M$, an open and closed neighborhood $\widetilde{\widetilde{X}}_0^{(j)}$ of $x_0^{(j)}$ in $X^{(j)}$ such that $\pi_{i,x^{(j)}}^{(j)}$ is unlinked for each $i=1,\dotsc,N$ and all $x^{(j)}\in\widetilde{\widetilde{X}}_0^{(j)}\cap\Sigma^{(j)}$. Then Theorem \ref{thmfamilythm} gives us a possibly smaller open and closed neighborhood $\widetilde{X}_0^{(j)}$ of $x_0^{(j)}$ in $\widetilde{\widetilde{X}}_0^{(j)}$ where, for all $x^{(j)}\in\widetilde{X}_0^{(j)}\cap\Sigma^{(j)}$ we can write
\[\pi_{i,x^{(j)}}^{(j)}=Q(\Delta_{1,i}^{(j)}\otimes\chi_{1,i,x^{(j)}}^{(j)},\dotsc,\Delta_{r_i^{(j)},i}^{(j)}\otimes\chi_{r_i^{(j)},i,x^{(j)}}^{(j)}),\]
for some integers $r_i^{(j)}$ which do not depend on $x^{(j)}\in\widetilde{X}_0^{(j)}\cap\Sigma^{(j)}$, some segments $\Delta_{1,i}^{(j)},\dotsc,\Delta_{r_i^{(j)},i}^{(j)}$ which do not depend on such $x^{(j)}$, and some unramified characters $\chi_{1,i,x^{(j)}}^{(j)},\dotsc,\chi_{r_i^{(j)},i,x^{(j)}}^{(j)}$. Then by Proposition \ref{propredtoonegln}, there is an open and closed neighborhood, which we take to be our $X_0^{(j)}$, of $x_0^{(j)}$ in $\widetilde{X}_0^{(j)}$, such that for each $i=1,\dotsc,N$, the following holds: Letting $R_0^{(j)}$ be the ring given by
\[R_0^{(j)}=\sset{\phi|_{X_0^{(j)}}}{\phi\in R^{(j)}},\]
and $\mc{H}_i^{(j)}$ be the Hecke algebra for $GL_{n_i^{(j)}}(F_i^{(j)})$, there is a linear map $T_i^{(j)}:\mc{H}_i^{(j)}\to R_0^{(j)}$ such that the specialization $T_{i,x^{(j)}}^{(j)}$ of $T_i^{(j)}$ satisfies
\[T_{i,x^{(j)}}^{(j)}(f_i^{(j)})=\tr(f_i^{(j)}|\pi_{i,x^{(j)}}^{(j)}),\]
for every $x^{(j)}\in\Sigma\cap X_0^{(j)}$ and $f_i^{(j)}\in\mc{H}_i^{(j)}$, and such that $\mc{F}_i^{(j)}=(X_0^{(j)},\Sigma^{(j)}\cap X_0^{(j)},R_0^{(j)},T_i^{(j)})$ is a generically irreducible family smooth admissible representations of $GL_{n_i^{(j)}}(F_i^{(j)})$.

So we fix an $i\in\{1,\dotsc,N\}$ throughout the rest of the proof. For any $j=1,\dotsc,M$, and for a positive integer $\ell^{(j)}$ and a supercuspidal representation $\tau^{(j)}$, let $n^{(j)}(\tau^{(j)},\ell^{(j)})$ be the number of segments $\Delta_{k,i}^{(j)}$ which are of length $\ell^{(j)}$ and contain an unramified twist of $\tau^{(j)}$, and let $k_1^{(j)},\dotsc,k_{n^{(j)}(\tau^{(j)},\ell^{(j)})}^{(j)}$ be the corresponding indices $k$. Let $r_{\tau^{(j)}}$ be the irreducible $W_{F_i}$-representation associated with $\tau^{(j)}$. Fix for now the $\ell^{(j)}$ and $\tau^{(j)}$ so that $n^{(j)}(\tau^{(j)},\ell^{(j)})\ne 0$ for every $j$. For each $x^{(j)}$ in $X_0^{(j)}\cap\Sigma^{(j)}$, let
\begin{equation}
\label{eqnRtaudef}
R_{x^{(j)}}^{(j)}(\tau^{(j)},\ell^{(j)})=\bigoplus_{l=1}^{n^{(j)}(\tau^{(j)},\ell^{(j)})}r_{\tau^{(j)}}\otimes\Sp(\ell^{(j)})\otimes\chi_{k_l,i,x^{(j)}}^{(j)},
\end{equation}
where we view each $\chi_{k_l,i,x^{(j)}}^{(j)}$ as an unramified character of $W_{F_i}$ by local class field theory. Then by Lemma \ref{lemepstwists}, it follows that
\begin{multline}
\label{eqnepsofRtaus}
\epsilon(\otimes_{j=1}^M(\varrho^{(j)}\circ R_{x^{(j)}}^{(j)}(\tau^{(j)},\ell^{(j)})),s_i,\psi_i)\\
=\epsilon(\otimes_{j=1}^M(\varrho^{(j)}\circ R_{x_0^{(j)}}^{(j)}(\tau^{(j)},\ell^{(j)})),s_i,\psi_i)\prod_{l=1}^{n^{(j)}(\tau^{(j)},\ell^{(j)})}\chi_{k_l,i,x^{(j)}}^{(j)}(\varpi_i)^{m_{l,i}^{(j)}(\tau^{(j)},\ell^{(j)})},
\end{multline}
for some nonnegative integers $m_{l,i}^{(j)}(\tau^{(j)},\ell^{(j)})$ depending a priori on $l=1,\dotsc,n^{(j)}(\tau^{(j)},\ell^{(j)})$, where $\varpi_i$ is a uniformizer in $F_i$. But since the isomorphism class of $\varrho^{(j)}$ is invariant under conjugation of its variable by the Weyl group, it follows that the powers $m_{l,i}^{(j)}(\tau^{(j)},\ell^{(j)})$ do not depend on $l$, and so we may write $m_i^{(j)}(\tau^{(j)},\ell^{(j)})=m_{l,i}^{(j)}(\tau^{(j)},\ell^{(j)})$ for any $l$.
Moreover, if $I(\tau^{(j)})$ is the twisting index of $\tau^{(j)}$ and $\zeta$ is an $I(\tau^{(j)})$th root of unity, then any factor of $R_{x^{(j)}}^{(j)}(\tau^{(j)},\ell^{(j)})$ in its definition in \eqref{eqnRtaudef} is invariant under twisting by $\chi^{(\zeta)}$. Thus \eqref{eqnepsofRtaus} implies
\begin{equation*}
\prod_{l=1}^{n^{(j)}(\tau^{(j)},\ell^{(j)})}(\chi_{k_l,i,x^{(j)}}^{(j)}(\varpi_i))^{m_i^{(j)}(\tau^{(j)},\ell^{(j)})}
=\zeta^{m_i^{(j)}(\tau^{(j)},\ell^{(j)})}\prod_{l=1}^{n^{(j)}(\tau^{(j)},\ell^{(j)})}(\chi_{k_l,i,x^{(j)}}^{(j)}(\varpi_i))^{m_i^{(j)}(\tau^{(j)},\ell^{(j)})},
\end{equation*}
which implies that $m_i^{(j)}(\tau^{(j)},\ell^{(j)})$ is divisible by $I(\tau^{(j)})$.

We now apply Theorem \ref{thmtwists} with the symmetric monomial
\[(X_1\dotsb X_{n^{(j)}(\tau^{(j)},\ell^{(j)})})^{m_i^{(j)}(\tau^{(j)},\ell^{(j)})/I(\tau^{(j)})};\]
By that theorem, there are elements $f_1,\dotsc,f_m\in\mc{H}_i^{(j)}$ for some $m$ and a polynomial $Q_i^{(j)}(\tau^{(j)},\ell^{(j)})$ in $m$ variables such that
\begin{equation*}
Q_i^{(j)}(\tau^{(j)},\ell^{(j)})((\tr(f_1|\pi_{i,x^{(j)}}^{(j)})),\dotsc,\tr(f_m|\pi_{i,x^{(j)}}^{(j)}))
=\prod_{l=1}^{n^{(j)}(\tau^{(j)},\ell^{(j)})}(\chi_{k_l,i,x^{(j)}}^{(j)}(\varpi_i))^{m_i^{(j)}(\tau^{(j)},\ell^{(j)})}.
\end{equation*}
(Note that we may absorb the numbers $w_{k_1},\dotsc,w_{k_{n(\tau,\ell)}}$ from that theorem into a constant which multiplies the symmetric monomial above.) Thus, letting $\phi_i^{(j)}(\tau^{(j)},\ell^{(j)})\in R_i^{(j)}$ be the element
\begin{multline*}
\phi_i^{(j)}(\tau^{(j)},\ell^{(j)})=\epsilon(\otimes_{j=1}^M(\varrho^{(j)}\circ R_{x_0^{(j)}}^{(j)}(\tau^{(j)},\ell^{(j)})),s_i,\psi_i)\\
\times Q_i^{(j)}(\tau^{(j)},\ell^{(j)})(T_i^{(j)}(f_1),\dotsc,T_i^{(j)}(f_m)),
\end{multline*}
we have that
\begin{equation}
\label{eqnphiijx}
\phi_i^{(j)}(\tau^{(j)},\ell^{(j)})(x^{(j)})=\epsilon(\otimes_{j=1}^M(\varrho^{(j)}\circ R_{x^{(j)}}^{(j)}(\tau^{(j)},\ell^{(j)})),s_i,\psi_i).
\end{equation}

Now by the definitions, we have that the Weil--Deligne representation associated with $\pi_{i,x^{(j)}}^{(j)}$ is
\[\bigoplus_{\tau^{(j)},\ell^{(j)}}R_{x^{(j)}}^{(j)}(\tau^{(j)},\ell^{(j)}),\]
where the sum ranges over $\ell^{(j)}$ and $\tau^{(j)}$ (the latter up to unramified twist) such that $n^{(j)}(\tau^{(j)},\ell^{(j)})\ne 0$. Thus by Lemma \ref{lemepstwists} (a) and \eqref{eqnphiijx}, taking
\[\mc{E}_i^{(j)}=\prod_{\tau^{(j)},\ell^{(j)}}\phi_i^{(j)}(\tau^{(j)},\ell^{(j)})\]
finishes the proof of the theorem.
\end{proof}

\section{Remarks on $p$-adic families of automorphic representations}
\label{secremeigen}

We now finish by singling out a source of examples of the families just studied. These come from the eigenvarieties of the type considered by Urban \cite{urbanev}. Let $G$ be a reductive group over $\Q$ such that $G(\R)$ has discrete series. Let $S$ be a finite set of rational primes. Assume that for any $\ell\in S$, the group $G(\Q_\ell)$ is a product of groups $GL_n(F)$ for various finite extensions $F$ of $\Q_\ell$. For example, the group $G$ could be the restriction of scalars to $\Q$ of a unitary group over a totally real number field $E$ such that every place of $E$ above any $\ell\in S$ splits in the CM field used to define $G$.

Now let $p\notin S$ be a rational prime, and fix an isomorphism $\C\cong\overline{\Q}_p$. Let
\[K^{S\cup\{p\}}=\prod_{\ell\notin S\cup\{p\}}G(\Z_\ell).\]
Consider the spherical Hecke algebra $\mc{H}_{\mr{sph}}^{S}$ with coefficients in $\Q_p$ given by
\[\mc{H}_{\mr{sph}}^{S}=C_c^\infty(K^{S\cup\{p\}}\backslash G(\A_f^{S\cup\{p\}})/K^{S\cup\{p\}},\Q_p),\]
where $\A_f^S$ denotes the finite adeles of $\Q$ away from $S$. Let
\[\mc{H}_S^p=C_c^\infty({\textstyle\prod_{\ell\in S}}G(\Q_\ell),\Q_p),\]
which is the full Hecke algebra over $\Q_p$ at the places in $S$.

Then, using the theory of eigenvarieties, one can often construct affinoid rigid analytic spaces $\mf{X}$ over $\Q_p$ along with the following. First, there is a $\Z_p$-algebra $\mc{U}_p$ of $U_p$-operators (see Urban \cite[\S 4.1.1]{urbanev}) and a $\Q_p$-linear map
\[I:\mc{H}_S^p\otimes_{\Q_p}\mc{H}_{\mr{sph}}^{S}\otimes_{\Z_p}\mc{U}_p\to\mc{O}(\mf{X}),\]
where $\mc{O}(\mf{X})$ is the ring of rigid analytic functions on $\mf{X}$. The map $I$ has the property that for any $x\in\mf{X}(\overline{\Q}_p)$, the specialization $I_x$ of $I$ at $x$ (that is, the composition of the $\overline\Q_p$-point $x$ with $I$) is the trace of a $p$-stabilization (see Urban \cite[\S 4.1.9]{urbanev}) of a smooth admissible $\sigma_x$ representation of $G(\A_f)$ (here $\A_f$ denotes the full finite adeles of $\Q$); moreover $\sigma_x$ is irreducible outside a Zariski closed subset of $\mf{X}$, and there is a Zariski dense subset of points $x$ in $\mf{X}(\overline{\Q}_p)$ where $\sigma_x$ is the $p$-stabilization of a cuspidal automorphic representation of $G$ which is discrete series at infinity of regular weight.

Now consider the ring $R=\mc{O}(\mf{X})\otimes_{\Q_p}\overline\Q_p$, and let $X=\mspec(R)$ be its maximal spectrum equipped with its Zariski topology. Then there is a surjective map $\mf{X}(\overline\Q_p)\to X$ (which identifies Galois conjugate points), so elements in $R$ may be naturally viewed as functions from $\mf{X}(\overline\Q_p)$ to $\overline{\Q}_p$.

The Hecke algebras $\mc{H}_{\mr{sph}}^{S}$ and $\mc{U}_p$ have identity elements. We may therefore consider the map
\[T:\mc{H}_S^p\otimes_{\Q_p}\overline\Q_p\to\overline\Q_p,\qquad T(f\otimes a)=aI(f\otimes 1\otimes 1).\]
Let $\Sigma$ be the subset of points $x$ in $\mf{X}(\overline\Q_p)$ where $I_x$ is the trace of the $p$-stabilization of an irreducible smooth representation of $G(\A_f)$. Then $(X,\Sigma,R,T)$ is a generically irreducible family of smooth admissible representations of $\prod_{\ell\in S}G(\Q_\ell)$; recall we have identified $\overline\Q_p$ with $\C$. The tuple $(X,\Sigma,R,T)$ is indeed such a family since for any analytic function $\phi\in\O(\mf{X})$, the zero locus of $\phi$ is a closed rigid analytic subspace of $\mf{X}$; thus the coherence condition (i) of Definition \ref{deffamily} is indeed satisfied.

Our Theorem \ref{thmfamilythm} then tells us how the local components at bad places of the $p$-adic automorphic representations parametrized by $\mf{X}$ vary.

Depending on the situation, one often knows that a cuspidal automorphic representation of $G(\A)$ which is discrete series at infinity of sufficiently regular weight is essentially tempered; for example in the case where $G$ is a unitary group as mentioned above, one knows this by work of Shin \cite{shin}. In these cases, by Proposition \ref{proptempunl}, the conclusion of Theorem \ref{thmepsilon} is then that the local epsilon factors of the members of the family that are $p$-stabilizations of irreducible smooth representations of $G(\A_f)$ which are unramified at $p$ vary analytically in the family; this seems to be a folklore conjecture in general, and so our theorem implies this folklore conjecture, for instance, in the case of a unitary group $G$ over a totally real field $E$ when every place of $E$ above any $\ell\in S$ splits in the CM field used to define $G$. In particular, the signs of the global functional equations of the $L$-functions of the members of such families must be Zariski locally constant, as they are thus analytic, being the products of local epsilon factors, and are equal to $1$ or $-1$.

But Theorem \ref{thmepsilon} can be applied in other situations as well. For example, another case of potential interest is the sign the triple product $L$-function attached to three modular eigenforms. Let $F_1,F_2,F_3$ be three cuspidal Coleman families of modular forms for $GL_2$. One can attach to these families various \textit{triple product }$p$-\textit{adic }$L$-\textit{functions} (\cite{AItriple, GS, Hsiehtriple}). Our Theorem \ref{thmepsilon} then can be used to show that these $L$-functions at their classical points of interpolation in the triple product family $F_1\times F_2\times F_3$ have functional equations with the same sign away from a Zariski closed subset of weight space. We remark that showing this rigorously would require translating the notion of Coleman family into the setting of this paper, but this is not too difficult in the context of \cite{urbanev}.

\printbibliography
\end{document}